\documentclass[a4paper,reqno]{amsart}
\numberwithin{equation}{section}
\usepackage{latexsym}
\usepackage{a4wide}
\usepackage{amscd}
\usepackage{graphics}
\usepackage{subcaption}
\usepackage{amsmath}
\usepackage{amssymb}
\usepackage{bm}
\usepackage{mathrsfs}
\usepackage{enumerate}
\usepackage{hyperref}
\usepackage{pgf,tikz}
\usepackage[utf8]{inputenc}

\usepackage{amsthm}
\newtheorem{thm}{Theorem}[section]

\newtheorem{lem}[thm]{Lemma}
\newtheorem{cor}[thm]{Corollary}

\newtheorem{prop}[thm]{Proposition}

\newtheorem{rem}[thm]{Remark}

\newcommand{\R}{\mathbb{R}}
\newcommand{\N}{\mathbb{N}}
\newcommand{\sfera}{\mathbb{S}}

\begin{document}

\title[Unique
  continuation from the edge of a crack]{Unique
  continuation from the edge of a crack}
\author{Alessandra De Luca and Veronica Felli}

\address{
\hbox{\parbox{5.7in}{\medskip\noindent
  A. De Luca and V. Felli\\
Dipartimento di Matematica e Applicazioni,
 Universit\`a di Milano - Bicocca, 
Via Cozzi 55, \\
20125 Milano (Italy). 
         {\em{E-mail addresses: }}{\tt a.deluca41@campus.unimib.it, veronica.felli@unimib.it.}}}}

\date{Revised version, June 2020}

\begin{abstract}
   In this work we develop an Almgren type monotonicity formula
  for a class of elliptic equations in a domain with a crack, in the
  presence of potentials satisfying either a negligibility condition
  with respect to the inverse-square weight or some suitable
  integrability properties.  The study of the Almgren frequency
  function around a point on the edge of the crack, where the domain
  is highly non-smooth, requires the use of an approximation argument,
  based on the construction of a sequence of regular sets which
  approximate the cracked domain.  Once a finite limit of the Almgren
  frequency is shown to exist, a blow-up analysis for scaled solutions
  allows us to prove asymptotic expansions and strong unique
  continuation from the edge of the crack.
       \end{abstract}
	
	\maketitle
	


\section{Introduction and statement of the main results}\label{sec:intr}
This paper presents a monotonicity approach to the study of the asymptotic behavior and unique continuation from the edge of a crack for solutions to the following class of elliptic equations 
\begin{equation}\label{eqgenerale}
\left\{\begin{aligned}
-\Delta u(x)&=f(x)u(x) &&\text{in $\Omega\setminus \Gamma$}, \\
 u&=0 && \text{on}\ \Gamma,
\end{aligned}\right. 
\end{equation}
where $\Omega\subset \R^{N+1}$ is a bounded open domain,
$\Gamma\subset \R^N$ is a closed set, $N\geq 2$, and the 
potential $f$ satisfies  either a negligibility condition with respect to
the inverse-square weight, see assumptions \eqref{xi0}-\eqref{xi}, or   some suitable integrability properties, see assumptions \eqref{eta0}-\eqref{eta} below. 

We recall that the \emph{strong unique continuation property} is said
to hold
for a certain class of equations if no solution, besides possibly the zero function, has a zero of infinite order.
Unique continuation principles for solutions to second
order elliptic equations have been largely studied in the literature
since the pioneering 
contribution by Carleman \cite{Carleman}, who derived unique
continuation from some weighted a priori inequalities. 
Garofalo and Lin in \cite{Garofalo} 
studied  unique
continuation for  elliptic equations with variable coefficients
introducing  an approach based on the validity of doubling conditions, 
which in turn depend on the monotonicity property of the Almgren type
frequency function, defined as the ratio 
 of scaled local energy over mass of the solution near a fixed point,
 see \cite{almgren}.

Once a strong unique continuation property is established and 
infinite
vanishing order for non-trivial solutions is excluded, the problem of estimating
and possibly  classifying all 
admissible vanishing rates naturally arises. 
For quantitative uniqueness and 
bounds for the maximal order of vanishing
obtained by monotonicity methods  we cite e.g. \cite{Kukavica};  
furthermore, a precise description of the asymptotic behavior together with a 
 classification of possible vanishing orders of solutions was
obtained for several classes of problems in \cite{Almgren-type,
  Semilinear, Asymptotic, MilanJ, Lemmini}, by combining monotonicity
methods with blow-up analysis for scaled solutions. 

The problem of unique continuation from boundary points presents
peculiar additional difficulties, as the derivation of monotonicity formulas is
made more delicate by the interference with the geometry of the
domain. Moreover the possible vanishing orders of solutions are
affected by the regularity of the boundary; e.g. in
\cite{Almgren-type} the asymptotic behavior at conical singularities
of the boundary has been shown to depend of the opening of the vertex.
We cite \cite{Adolf-Escau, Adolf-Kenig, Almgren-type,
  Kukavica-Nystrom, Tao} for unique continuation from the boundary for
elliptic equations under homogeneous Dirichlet conditions. We also
refer to \cite{TZ05} for unique continuation and doubling
properties at the boundary under zero Neumann conditions
and to \cite{DFV} for a strong unique continuation result from the
vertex of a cone under non-homogeneous Neumann conditions.

The aforementioned papers concerning unique continuation from the
boundary  require the domain to
be at least of Dini type. With the aim of relaxing this kind of
regularity assumptions, the present paper
investigates unique continuation and  classification of the possible vanishing
orders 
of solutions at edge points of cracks breaking the domain, which are
then highly irregular points of the boundary.

Elliptic problems in domains with cracks arise in elasticity theory, 
see e.g. \cite{DalMaso, KLH, Lazzaroni}. 
  The high  non-smoothness of
 domains with slits produces strong singularities of solutions to elliptic
problems at edges of cracks; the structure of such singularities 
 has been widely studied in the
literature, see e.g. \cite{CD, Costabel,DW} and references therein.
 In particular, asymptotic expansions of solutions at
 edges play a crucial role in crack problems, since  the coefficients
 of such expansions are related to the so called \textit{stress
   intensity factor}, see   e.g. \cite{DalMaso}. 

 A further reason of interest in the study of problem
 \eqref{eqgenerale} can be found in its relation with mixed
 Dirichlet/Neumann boundary value problems. Indeed, if we consider an
 elliptic equation associated, on  a flat portion of the boundary
 $\Lambda=\Lambda_1\cup \Lambda_2$, to 
a homogeneous Dirichlet boundary condition on $\Lambda_1$ and a homogeneous Neumann condition on
 $\Lambda_2$, an even reflection through the flat boundary $\Lambda$ leads to
 an elliptic equation satisfied in the complement of the Dirichlet
 region, which then plays the role of a crack, see Figure \ref{fig:mixed}; the edge of the crack
 corresponds to the Dirichlet-Neumann junction of the original
 problem. In \cite{Fall} unique continuation and asymptotic expansions
 of solutions for planar mixed boundary value problems at
 Dirichlet-Neumann junctions were obtained via monotonicity methods;
 the present paper is in part motivated by the aim of extending to
 higher dimensions the monotonicity formula obtained in \cite{Fall}
 in the 2-dimensional case, together with its applications to unique continuation.  For some regularity
 results for second-order elliptic problems with mixed
 Dirichlet-Neumann type boundary conditions we refer to \cite{Kassman,
   Savare} and references therein.

\begin{figure}[b]
  \centering
  \begin{subfigure}[nooneline]{0.3\linewidth}
  \begin{tikzpicture}[scale=0.4]
 \draw[fill=black, fill opacity=0.2,rounded corners=1.5ex]
(4,0)  -- (3.4,0) -- (2,0.1) -- 
(0.2,1)  --
(0,3)  --
(3,6)  --
(5,4)  --
(8,3)  --
(9,2)  --
(8.8,0) -- (4,0); 
 \draw[color=blue,line width=1pt](3.2,0) -- (6,0);
  \draw[color=blue] (4.4,0.4) node {\tiny Neumann};
  \draw[color=blue] (4.4,-0.4) node {\tiny $\Lambda_2$};
  \draw[color=red,line width=1pt](6,0) -- (8.4,0);
 \draw[color=red] (7.2,0.4) node {\tiny Dirichlet};
 \draw[color=red] (7.2,-0.4) node {\tiny $\Lambda_2$};
\end{tikzpicture}
\vskip2.1cm
 \caption{Mixed Dirichlet/Neumann boundary conditions on
 a flat portion of the boundary}
  \end{subfigure}
  \qquad\begin{subfigure}[nooneline]{0.3\linewidth}
  \begin{tikzpicture}[scale=0.4]
 \draw[fill=black, fill opacity=0.2,rounded corners=1.5ex]
(3.4,0) -- (2,0.1) -- 
(0.2,1)  --
(0,3)  --
(3,6)  --
(5,4)  --
(8,3)  --
(9,2)  --
(8.8,0) -- (7,0);
 \draw[fill=black, fill opacity=0.2,rounded corners=1.5ex]
(3.4,0) -- (2,-0.1) -- 
(0.2,-1)  --
(0,-3)  --
(3,-6)  --
(5,-4)  --
(8,-3)  --
(9,-2)  --
(8.8,0) -- (7,0); 
  \draw[color=red,line width=1pt](6,0) -- (8.4,0);
 \draw[color=red] (7.2,0.4) node {\tiny Crack};
\end{tikzpicture}
 \caption{After an even reflection the Dirichlet region becomes a crack}
\end{subfigure}
  \caption{A motivation from mixed Dirichlet/Neumann boundary value problems}
  \label{fig:mixed}
\end{figure}
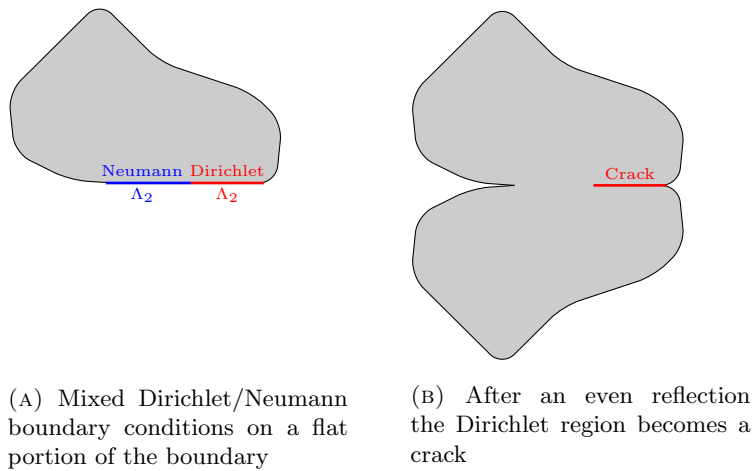

 In the generalization of the Almgren type monotonicity formula of
 \cite{Fall} to dimensions greater than 2, some new additional
 difficulties arise, besides the highly non-smoothness of the domain:
 the positive dimension of the edge, a stronger interference with the
 geometry of the domain, and some further technical issues, related
 e.g. to the lack of conformal transformations straightening the edge.
In particular,  the proof of the monotonicity formula is based on the
differentiation of the Almgren quotient defined in
\eqref{frequency}, which in turn requires a 
Pohozaev  type identity formally obtained by testing  the equation with the function
$\nabla u\cdot x$; however  our  domain with crack  doesn't
verify the  exterior ball
condition (which ensures $L^2$-integrability of second order
derivatives, see \cite{Adolf}) and  $\nabla u\cdot x$ could be not
sufficiently regular to be an admissible
test function.

 In this article a new technique, based on an
  approximation argument, is developed to overcome the aforementioned difficulty:
  we construct first a sequence of domains  which approximate
  $\Omega\setminus\Gamma$, satisfying the exterior ball condition and
  being star-shaped with respect to the origin, and then a sequence of
  solutions of an approximating problem on such domains, converging to
  the solution of the original problem with crack. For the
  approximating problems enough regularity is available to establish a
Pohozaev  type identity, with some remainder terms due to interference
with the boundary, whose sign can nevertheless be recognized thanks to
  star-shapeness conditions. Then, passing to the limit in Pohozaev  identities
for the approximating problems, we obtain inequality \eqref{pohoz},
which is enough to estimate from below the derivative of the Almgren
quotient and to prove that such quotient has a finite limit at
$0$ (Lemma \ref{Lemlimesiste}).  Once a finite  limit of the Almgren
frequency is shown to exist, a blow-up analysis for scaled solutions
allows us to prove strong unique continuation and asymptotics of
solutions. 

In order to state the main results of the present paper, we start by
introducing our assumptions on the domain. For $N\geq 2$, we consider the set 
\[
\Gamma=\{(x',x_N)=(x_1,\dots,x_{N-1},x_N)\in \mathbb{R}^N:\, x_N\geq g(x')\},
\]
where $g\colon\mathbb{R}^{N-1}\to \mathbb{R}$ is a function such that 
\begin{equation}\label{g0=0}
g(0)=0, \quad  \nabla g(0)=0,
\end{equation}
\begin{equation}\label{gC2}
g\in C^2(\mathbb{R}^{N-1}).
\end{equation}
Let us observe that assumption \eqref{g0=0} is not a restriction but
just a selection of our coordinate system and, from \eqref{g0=0} and
\eqref{gC2}, it follows that
\begin{equation}\label{gO2}
|g(x')|=O(|x'|^2)\quad \text{as $|x'|\rightarrow 0^+$}.
\end{equation}
Moreover we assume that 
\begin{equation}\label{nablaggeq0}
g(x')-x'\cdot\nabla g(x')\geq 0
\end{equation}
for any $x'\in B'_{\hat{R}}:=\{x'\in \mathbb{R}^{N-1}:|x'|<\hat{R}\}$, for
some $\hat{R}>0$. This condition says that
$\overline{\mathbb{R}^N\setminus\Gamma}$ is star-shaped with respect
to the origin in a neighbourood of 0. It is satisfied for instance if
the function $g$ is concave in a neighborhood of the origin.

We are interested in studying the following boundary value problem
\begin{equation}\label{eq:CPu}
\left\{\begin{aligned}
-\Delta u&=f\,u &&\text{in}\ B_{\hat{R}}\setminus\Gamma, \\
u&=0 && \text{on}\ \Gamma,
\end{aligned}\right. 
\end{equation}
where $B_{\hat{R}}=\{x\in \R^{N+1}:|x|<\hat R\}$, for some function $f\colon B_{\hat{R}}\to \mathbb{R}$ such that $f$ is measurable and bounded in $B_{\hat{R}}\setminus B_\delta$ for every $\delta\in (0, \hat{R})$. We consider two alternative sets of assumptions: we assume either that 
\begin{equation}\label{xi0} \tag{H1-1}
\lim _{r\rightarrow 0^+}\xi_f(r)=0,
\end{equation}
\begin{equation}\label{xiL1}\tag{H1-2}
\frac{\xi_f(r)}{r}\in L^1(0,\hat{R}),\quad \frac{1}{r}\int_0^r\frac{\xi_f(s)}{s}\,ds\in L^1(0,\hat{R}),
\end{equation}
where the function $\xi_f$ is defined as
\begin{equation}\label{xi}\tag{H1-3}
\xi_f(r):=\sup_{x\in \overline{B_r}}|x|^2|f(x)|\quad \text{for any $r\in (0,\hat{R})$},
\end{equation}
or that 
\begin{equation}\label{eta0}\tag{H2-1}
\lim _{r\rightarrow 0^+}\eta(r,f)=0,
\end{equation}
\begin{equation}\label{etaL1}\tag{H2-2}
\frac{\eta(r,f)}{r}\in L^1(0,\hat{R}),\quad \frac{1}{r}\int_0^r\frac{\eta(s,f)}{s}\,ds\in L^1(0,\hat{R}),
\end{equation}
and 
\begin{equation}\label{etanablaf}\tag{H2-3}
\nabla f\in L^\infty_{\mathrm{loc}}(B_{\hat{R}}\setminus \{0\}),
\end{equation} 
\begin{equation}\label{etanablafL1}\tag{H2-4}
\frac{\eta(r,\nabla f\cdot x)}{r}\in L^1(0,\hat{R}),\quad \frac{1}{r}\int_0^r\frac{\eta(s,\nabla f \cdot x)}{s}\,ds\in L^1(0,\hat{R}),
\end{equation}
where 
\begin{equation}\label{eta}\tag{H2-5}
\eta(r,h)=\sup_{u\in H^1(B_r)\setminus \{0\}}\frac{\int_{B_r}|h|u^2\,dx}{\int_{B_r}|\nabla u|^2\,dx+\frac{N-1}{2r}\int_{\partial B_r}|u|^2\, dS},
\end{equation}
for every $r\in (0,\hat{R}),\,h\in L^\infty_{\mathrm{loc}}(B_{\hat{R}}\setminus \{0\})$.

Conditions \eqref{xi0}-\eqref{xi} are satisfied e.g. if
$|f(x)|=O(|x|^{-2+\delta})$ as $|x|\rightarrow 0$ for some $\delta>0$,
whereas assumptions \eqref{eta0}-\eqref{eta} hold e.g. if
$f\in W^{1, \infty}_{\mathrm{loc}}(B_{\hat{R}}\setminus\{0\})$ and
$f,\nabla f\in L^p(B_{\hat{R}})$ for some $p>\frac{N+1}{2}$.
 We also observe that condition \eqref{eta0} is satisfied
  if $f$ belongs to the Kato class $K_{n+1}$, see \cite{FGL}.

In order to give a weak formulation of problem \eqref{eq:CPu}, we introduce the space $H^1_{\Gamma}(B_R)$ for every $R>0$, defined as the closure in $H^1(B_R)$ of the subspace 
\[
C^\infty_{0,\Gamma}(\overline{B_R}):=\{u\in
C^\infty(\overline{B_R}):u=0\ \text{in a neighborhood of}\ \Gamma\}.
\]
We observe that actually 
\begin{equation*}
H^1_\Gamma(B_R)=\{u\in H^1(B_R):\, \tau_\Gamma(u)=0\},
\end{equation*}
where $\tau_{\Gamma}$ denotes the trace operator on $\Gamma$, as one
can easily deduce from \cite{Bern}, taking into account that the capacity of $\partial\Gamma$ in $\mathbb{R}^{N+1}$ is zero, since $\partial\Gamma$ is contained in a 2-codimensional manifold. 

Hence we say that $u \in H^1(B_{\hat{R}})$ is a weak solution to \eqref{eq:CPu} if
\begin{equation*}\label{eq:CPudeb}
\left\{\begin{aligned}
&u\in H^1_\Gamma(B_{\hat{R}}), \\
&\int_{B_{\hat{R}}}\nabla u(x)\cdot\nabla v(x)\,dx-\int_{B_{\hat{R}}}f(x)u(x)v(x)\,dx=0\quad\text{for any}\ v\in C_c^\infty(B_{\hat{R}}\setminus\Gamma).
\end{aligned}\right. 
\end{equation*}
 In the classification of the possible vanishing orders and
blow-up profiles  of
solutions, the following
eigenvalue problem on the unit $N$-dimensional sphere with a
half-equator cut plays a crucial role. 
Letting
$\mathbb{S}^N=\{(x',x_N,x_{N+1}):|x'|^2+x_N^2+x_{N+1}^2=1\}$ be the unit
$N$-dimensional sphere and $\Sigma=\{(x',x_N,x_{N+1})\in
\mathbb{S}^N:x_{N+1}=0\text{ and }x_N\geq0\}$, we consider the
eigenvalue problem 
\begin{equation}\label{eq:1}
\left\{\begin{aligned}
-\Delta_{\sfera^N}\psi&=\mu\,\psi&&\text{on }\sfera^N\setminus\Sigma,\\
\psi&=0&&\text{on }\Sigma.
  \end{aligned}\right.
\end{equation}
We say that $\mu\in\R$ is an eigenvalue of  \eqref{eq:1} if
there exists an eigenfunction $\psi\in
H^1_0(\sfera^N\setminus\Sigma)$, $\psi\not\equiv0$, such that  
\[
\int_{\sfera^N}\nabla_{\sfera^N}\psi\cdot \nabla_{\sfera^N}\phi\,dS=\mu
\int_{\sfera^N}\psi\phi\,dS
\]
for all $\phi\in H^1_0(\sfera^N\setminus\Sigma)$. By classical spectral
theory, \eqref{eq:1} admits   a diverging sequence of real eigenvalues
with finite multiplicity $\{\mu_k\}_{k\geq1}$; moreover these
eigenvalues are
explicitly  given by the formula
\begin{equation}\label{eq:2}
\mu_k=\frac{k(k+2N-2)}4,\quad k\in \mathbb{N}\setminus\{0\},
\end{equation}
see Appendix \ref{sec:app_A}.
 For all $k\in \mathbb{N}\setminus\{0\}$, let $M_k\in\N\setminus\{0\}$ be the
multiplicity of the eigenvalue $\mu_k$ and 
\begin{equation}\label{eq:32}
  \begin{split}
    \{Y_{k ,m}\}_{m=1,2,\dots ,M_{k}}\text{ be a
      $L^{2}(\mathbb{S}^{N})$-orthonormal basis}\\\text{ of the eigenspace of
      \eqref{eq:1} associated to $\mu_{k}$}.
  \end{split}
\end{equation}
 In particular $\{Y_{k,m}:k\in
\mathbb{N}\setminus\{0\}, m=1,2,\dots,M_k\}$ is an orthonormal basis
of $L^{2}(\mathbb{S}^{N})$.

The main result of this paper provides an evaluation of the behavior at 0 of weak solutions $u\in H^1(B_{\hat{R}})$ to the boundary value problem \eqref{eq:CPu}.
\begin{thm}\label{teoremagenerale}
  Let $N\geq 2$ and $u\in H^1(B_{\hat{R}})\setminus \{0\}$ be a
  non-trivial weak solution to \eqref{eq:CPu}, with $f$ satisfying
  either assumptions \eqref{xi0}-\eqref{xi} or
  \eqref{eta0}-\eqref{eta}. Then, there exist
  $k_0\in\mathbb{N}$, $k_0\geq 1$, and an eigenfunction  of problem \eqref{eq:1}
associated with the eigenvalue $\mu_{k_0}$
such that
\begin{equation}\label{lambda^gammaesplicito}
  \lambda^{-k_0/2}u(\lambda x)\rightarrow |x|^{k_0/2}\psi(x/|x|)\quad 
  \text{as $\lambda \rightarrow 0^+$}
\end{equation}
in $H^1(B_1)$. 
\end{thm}
We mention that a strongest version of Theorem \ref{teoremagenerale}
will be given in Theorem \ref{teoremabetai}.

As a direct consequence of Theorem \ref{teoremagenerale}
and the boundedness of eigenfunctions of \eqref{eq:1} (see 
Appendix \ref{sec:app_A}),
the following point-wise upper bound holds.
\begin{cor}
  Under the same assumptions as in Theorem \ref{teoremagenerale}, let
  $u\in H^1(B_{\hat{R}})$ be a non-trivial weak solution to
  \eqref{eq:CPu}. Then, there exists $k_0\in\mathbb{N}$, $k_0\geq 1$,
  such that
\[
u(x)=O(|x|^{k_0/2})\quad \text{as $|x|\rightarrow 0^+$}.
\]
\end{cor} 

We observe that, due to the vanishing on the half-equator $\Sigma$ of the angular profile $\psi$
appearing in \eqref{lambda^gammaesplicito}, we cannot expect the
reverse estimate $|u(x)|\geq c|x|^{k_0/2}$ to hold for $x$ close to the
origin.  

A further relevant consequence of our asymptotic analysis is the following unique continuation principle, whose proof follows straightforwardly from Theorem \ref{teoremagenerale}.
\begin{cor}
Under the same assumptions as in Theorem \ref{teoremagenerale}, let $u\in H^1(B_{\hat{R}})$ be a weak solution to \eqref{eq:CPu} such that $u(x)=O(|x|^k)$ as $|x|\rightarrow 0$, for any $k\in\mathbb{N}$. Then $u\equiv 0$ in $B_{\hat{R}}$.
\end{cor}
Theorem \ref{teoremabetai} will actually give a more precise
description on the limit angular profile $\psi$: if $M_{k_0}\geq 1$
is the multiplicity of the eigenvalue $\mu_{k_0}$ and $\{Y_{k_0,i} :
1\leq i\leq M_{k_0}\}$  is as in \eqref{eq:32}, 
then the eigenfunction $\psi$ in \eqref{lambda^gammaesplicito} can be written as 
\begin{equation}\label{eq:4}
\psi(\theta)=\sum_{i=1}^{m_{k_0}}\beta_i Y_{k_0,i},
\end{equation}
where the coefficients $\beta_i$ are given by the \textit{integral Cauchy-type formula} \eqref{betaiespliciti}.

The paper is organized as follows. In Section 
\ref{sec:appr-probl} 
 we construct a sequence of problems on smooth sets approximating
 the cracked domain,  with corresponding solutions converging to
  the solution of problem \eqref{eq:CPu}. In Section \ref{sec:pohozaev-identity}
we derive  a
Pohozaev  type identity for the approximating problems and
consequently inequality \eqref{pohoz}, which is then used in Section
\ref{sec:almgr-type-freq} to prove the existence of the limit for the
Almgren type 
quotient associated to problem \eqref{eq:CPu}. In Section 
\ref{sec:blow-up-argument} we perform a blow-up analysis and prove
that scaled
solutions converge in some suitable sense to a homogeneous limit
profile, whose homogeneity order is related to the eigenvalues of
problem \eqref{eq:1} and whose angular component is shown to be as in 
\eqref{eq:4} in Section \ref{sec:straightening-domain}, where an auxiliary
equivalent problem with a straightened crack is constructed.
Finally, in the appendix we 
 derive the explicit formula \eqref{eq:2} for the
eigenvalues of problem~\eqref{eq:1}.

\medskip
\noindent \textbf{Notation.} We list below some notation used throughout the paper.
\begin{itemize}
\item[-] For all $r>0$, $B_r$ denotes the open ball $\lbrace x=(x',x_N,x_{N+1})\in \mathbb{R}^{N+1}: |x|<r\rbrace$ in $\mathbb{R}^{N+1}$ with radius $r$ and center at 0.
\item[-] For all $r>0$, $\overline{B_r}=\lbrace x=(x',x_N,x_{N+1})\in \mathbb{R}^{N+1}:|x|\leq r \rbrace$ denotes the closure of $B_r$.
\item[-] For all  $r>0$, $B'_r$ denotes the open ball $\lbrace x=(x',x_N)\in \mathbb{R}^N:|x|<r\rbrace$ in $\mathbb{R}^N$ with radius $r$ and center at 0.
\item[-] $dS$ denotes the volume element on the spheres $\partial B_r$, $r>0$.\\
\end{itemize}

\section{Approximation problem}\label{sec:appr-probl}

 We first prove a coercivity type result for the
quadratic form associated to problem \eqref{eq:CPu} in small
neighbourhoods of $0$.  

\begin{lem}\label{lemmacoerc}
Let $f$ satisfy either \eqref{xi0} or
\eqref{eta0}. Then there exists $r_0\in(0,\hat{R})$ such that, for any $r\in(0,r_0]$ and $u\in H^1(B_r)$,
\begin{equation}\label{equazcoerc}
\int_{B_r}(|\nabla u|^2-|f|u^2)\,dx\geq \frac{1}{2}\int_{B_r}|\nabla u|^2\,dx -\omega(r)\int_{\partial B_r}u^2\,dS
\end{equation}
 and 
\begin{equation}\label{eq:bound-omega}
  r\omega(r)< \frac{N-1}4,
\end{equation}
where 
\begin{equation}\label{eq:omega}
\omega(r)=
\begin{cases}
\dfrac{2}{N-1}\dfrac{\xi_f(r)}{r}, &\text{under assumption
\eqref{xi0}},\\[10pt]
\dfrac{N-1}{2}\dfrac{\eta(r,f)}{r}, &\text{under
assumption \eqref{eta0}}.
\end{cases}
\end{equation}
\end{lem}
\begin{rem}
 For future reference, it is useful to rewrite \eqref{equazcoerc}
as 
\begin{equation}\label{equazcoercriscritta}
\int_{B_r}|f|u^2\,dx\leq \frac{1}{2}\int_{B_r}|\nabla u|^2\,dx +\omega(r)\int_{\partial B_r} u^2\,ds
\end{equation}
for all $u\in H^1(B_r)$ and $r\in (0,r_0]$.
\end{rem}
The proof of Lemma \ref{lemmacoerc} under assumption \eqref{xi0}
is based on the following 
Hardy type inequality with boundary terms, due to Wang and Zhu \cite{Wang}.
\begin{lem}[\cite{Wang}, Theorem 1.1]\label{lemmaHardy}
For every $r>0$ and $u\in H^1(B_r)$,
\begin{equation}\label{Hardy}
\int_{B_r}|\nabla u(x)|^2\,dx+\frac{N-1}{2r}\int_{\partial B_r}|u(x)|^2\,dS\geq\biggl(\frac{N-1}{2}\biggr)^2\int_{B_r}\frac{|u(x)|^2}{|x|^2}\,dx.
\end{equation}
\end{lem}

\begin{proof}[Proof of Lemma \ref{lemmacoerc}] Let us first prove the
  lemma under assumption
 \eqref{xi0}. To this purpose, let $r_0\in (0,\hat{R})$ be such that 
\begin{equation}\label{xir0}
\frac{4\xi_f({r})}{(N-1)^2}< \frac{1}{2}\quad\text{for all }r\in(0,r_0].
\end{equation}
Using the definition of $\xi_f(r)$ \eqref{xi} and \eqref{Hardy}, we have that for any $r\in (0,\hat{R})$ and $u\in H^1(B_r)$
\begin{equation}\label{fu^2}
\int_{B_r}|f|u^2\,dx\leq \xi_f(r)\int_{B_r}\frac{|u(x)|^2}{|x|^2}\,dx\leq\frac{4\xi_f(r)}{(N-1)^2}\biggl[\int_{B_r}|\nabla u|^2\,dx+\frac{N-1}{2r}\int_{\partial B_r}u^2\,dS\biggr].
\end{equation}
Thus, for every $0<r\leq r_0$, from \eqref{xir0} and \eqref{fu^2}, we obtain that
\begin{equation*}
\begin{split}
\int_{B_r}\big(|\nabla u|^2-|f|u^2\big)\,dx&\geq\biggl(1-\frac{4\xi_f(r)}{(N-1)^2}\biggr)\int_{B_r}|\nabla u|^2\,dx-\frac{2}{N-1}\frac{\xi_f(r)}{r}\int_{\partial B_r}u^2\,ds\\
&\geq \frac{1}{2}\int_{B_r}|\nabla u|^2\,dx-\frac{2}{N-1}\frac{\xi_f(r)}{r}\int_{\partial B_r}u^2\,ds
\end{split}
\end{equation*}
and this completes the proof of \eqref{equazcoerc}  under assumption
 \eqref{xi0}. 

Now let us prove the lemma under assumption \eqref{eta0}. Let $r_0\in(0,\hat{R})$ be such that 
\begin{equation}\label{<1/2}
\eta(r,f)< \frac{1}{2}\quad \text{for all $r\in(0,r_0]$}.
\end{equation}
From the definition of $\eta(r,f)$ \eqref{eta} it follows that for any $r\in (0,\hat{R})$ and $u\in H^1(B_r)$
\begin{equation}\label{fu^2'}
\int_{B_r}|f|u^2\,dx\leq \eta(r,f)\biggl[\int_{B_r}|\nabla u|^2\,dx+\frac{N-1}{2r}\int_{\partial B_r}u^2\,dS\biggr].
\end{equation}
Thus, for every $0<r\leq r_0$, from \eqref{<1/2} and \eqref{fu^2'} we deduce that
\begin{equation*}
\begin{split}
\int_{B_r}\big(|\nabla u|^2-|f|u^2\big)\,dx&\geq(1-\eta(r,f))\int_{B_r}|\nabla u|^2\,dx-\frac{N-1}{2}\frac{\eta(r,f)}{r}\int_{\partial B_r}u^2\,dS\\
&\geq \frac{1}{2}\int_{B_r}|\nabla u|^2\,dx-\frac{N-1}{2}\frac{\eta(r,f)}{r}\int_{\partial B_r}u^2\,ds,
\end{split}
\end{equation*}
hence concluding the proof of \eqref{equazcoerc} under assumption \eqref{eta0}.

We observe that estimate \eqref{eq:bound-omega} follows from the
definition of $\omega$ in \eqref{eq:omega}, \eqref{xir0}, and \eqref{<1/2}.
\end{proof}
Now we are going  to construct suitable  regular  sets which are
star-shaped with respect to the origin  and which approximate our
cracked domain.   In order to do this, for any $n\in \mathbb{N}\setminus\{0\}$ let $f_n\colon\mathbb{R}\to\mathbb{R}$ be defined as
\[f_n(t)=
\begin{cases}
n|t|+\frac{1}{n}e^{\frac{2n^2|t|}{n^2|t|-2}}, & \text{if } |t|<2/n^2,\\
n|t|, & \text{if }|t|\geq 2/n^2,
\end{cases}
\]
so that $f_n\in C^2(\mathbb{R})$, $f_n(t)\geq n|t|$ and $f_n$
increases for all $t>0$ and decreases for all $t<0$;  furthermore 
\begin{equation}\label{eq:5}
  f_n(t)-t\,f'_n(t)\geq 0\quad\text{for every }t\in\R.
\end{equation}
For all $r>0$ we define 
\begin{equation}\label{Btilde}
\tilde{B}_{{r},n}=\{(x',x_N, x_{N+1})\in B_{r}:\,x_N<g(x')+f_n(x_{N+1})\},
\end{equation}
see Figure \ref{fig:Approximating-domains}.
\begin{figure}[b]
  \centering
  \begin{subfigure}[nooneline]{0.3\linewidth}
    \includegraphics[scale = 0.4]{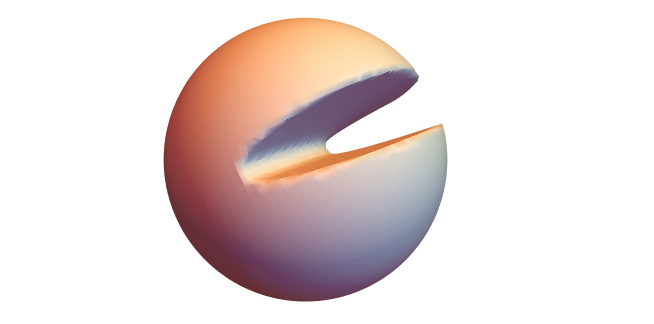}
    \caption{The set $\tilde{B}_{r,n}$}
  \end{subfigure}
  \begin{subfigure}[nooneline]{0.3\linewidth}
    \includegraphics[scale = 0.4]{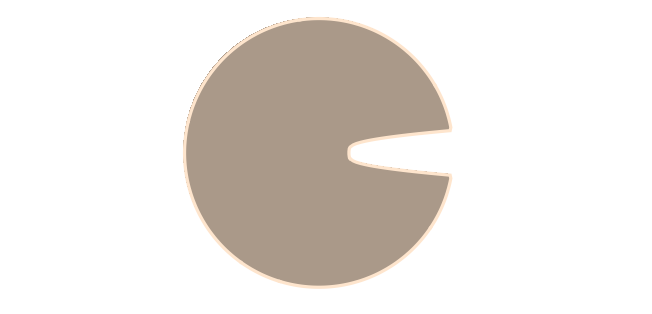}
   \caption{Section of $\tilde{B}_{r,n}$}
  \end{subfigure}
  \caption{Approximating domains}
  \label{fig:Approximating-domains}
\end{figure}
Let $\tilde{\gamma}_{r,n}\subset \partial \tilde{B}_{r,n}$ be the subset of $B_{r}$ defined as
\begin{equation*}
\tilde{\gamma}_{r,n}=\{(x',x_N, x_{N+1})\in B_{r}:\, x_N=g(x')+f_n(x_{N+1})\}
\end{equation*}
and $\tilde{S}_{r,n}$ denote the set given by
$\partial\tilde{B}_{r,n}\setminus \tilde{\gamma}_{r,n}$.
We note that, for any fixed $r>0$, the set
$\tilde{\gamma}_{r,n}$ is not empty and 
$\tilde{B}_{r,n}\neq B_r$
provided $n$ is sufficiently large.
\begin{lem}\label{l:stellato}
Let $0<r\leq\hat{R}$.  
Then, for all
$n\in{\mathbb N}\setminus\{0\}$,  the
  set $\tilde{B}_{r,n}$ is star-shaped with respect to the origin,
  i.e. $x\cdot \nu \geq 0$ for a.e. $x\in \partial \tilde{B}_{r,n}$,
  where $\nu$ is the outward unit normal vector.
\end{lem} 
\begin{proof}
If $\tilde{\gamma}_{r,n}$ is empty, then $\tilde{B}_{r,n}= B_r$
and the conclusion is obvious. Let $\tilde{\gamma}_{r,n}$ be not
empty. 

 The thesis is trivial if one considers a point
  $x\in \tilde{S}_{r,n}$.

  If $x\in \tilde{\gamma}_{r,n}$, then
  $x=(x', g(x')+f_n(x_{N+1}), x_{N+1})$ and the outward unit normal
  vector at this point is given by
\[
\nu(x)=\frac{(-\nabla g(x'), 1,
  -f'_n(x_{N+1}))}{\sqrt{1+|f'_n(x_{N+1})|^2+|\nabla g(x')|^2}},
\]  
hence we have that 
\[
x\cdot \nu(x)=\frac{g(x')-\nabla g(x')\cdot x'+f_n(x_{N+1})-x_{N+1}\
  f'_n(x_{N+1})}{\sqrt{1+|f'_n(x_{N+1}|^2+|\nabla g(x')|^2}}\geq 0
\]
since $g(x')-\nabla g(x')\cdot x'\geq 0$ by assumption
\eqref{nablaggeq0} and $f_n(x_{N+1})-x_{N+1}\,f'_n(x_{N+1})\geq 0$ by \eqref{eq:5}.
\end{proof}

 From now on, we fix $u\in
H^1(B_{\hat{R}})\setminus \{0\}$, a non-trivial weak solution to
problem \eqref{eq:CPu},  with $f$ satisfying
  either \eqref{xi0}-\eqref{xi} or
  \eqref{eta0}-\eqref{eta}. 
  Since $u \in H^1_\Gamma(B_{\hat{R}})$, there exists a sequence of functions $g_n\in C^\infty_{0,\Gamma}(\overline{B_{\hat{R}}})$ such that $g_n\rightarrow u$ in $H^1(B_{\hat{R}})$. We can choose the functions $g_n$ in such a way that 
\begin{equation}\label{gn=0}
g_n(x_1,\dots,x_N,x_{N+1})=0\quad \text{if $(x_1,\cdots,x_N)\in\Gamma$ and $|x_{N+1}|\leq\frac{\tilde{C}}{n}$},
\end{equation}
with 
\begin{equation}\label{eq:6}
\tilde{C}>\sqrt{2(r_0^2+M^2)}, \quad \text{where }
M=\mathrm{max}\{|g(x')|:|x'|\leq r_0\}.
\end{equation}
\begin{rem}\label{rem:g_n_zero}
  We observe that $g_n\equiv0$ in
  $B_{r_0}\setminus \tilde{B}_{r_0,n}$. Indeed, if
  $x=(x',x_N,x_{N+1})\in B_{r_0}\setminus \tilde{B}_{r_0,n}$, then
\[
x_N\geq g(x')+f_n(x_{N+1})>g(x'),
\]
so that $(x',x_N)\in\Gamma$. Moreover 
\[
x_N\geq f_n(x_{N+1})+g(x')\geq n|x_{N+1}|-M,
\]
with $M$ as in \eqref{eq:6}.
Hence either $|x_{N+1}|\leq \frac{M}{n}$ or $r_0^2\geq x_N^2\geq
(n|x_{N+1}|-M)^2\geq \frac{n^2}{2}|x_{N+1}|^2-M^2$. Thus
$|x_{N+1}|\leq \frac{\sqrt{2(r_0^2+M^2)}}{n}<\frac{\tilde{C}}{n}$, if
we choose $\tilde{C}$ as in \eqref{eq:6}. Then  $g_n(x)=0$ in view of \eqref{gn=0}.
\end{rem}
Now we construct a sequence of approximated solutions
$\{u_n\}_{n\in\mathbb{N}}$ on the sets $\tilde{B}_{r_0,n}$. For each
fixed $n\in\mathbb{N}$, we claim that there exists a unique weak
solution $u_n$ to the boundary value problem
\begin{equation}\label{eq:CPu_n}
\left\{\begin{aligned}
-\Delta u_n&=f u_n &&\text{in}\ \tilde{B}_{{r_0},n}, \\
u_n&=g_n && \text{on}\ \partial \tilde{B}_{{r_0},n}.
\end{aligned}\right. 
\end{equation}
Letting 
\[
v_n:=u_n-g_n,
\] 
we have that $u_n$ weakly solves \eqref{eq:CPu_n} if and only if $v_n\in H^1(\tilde{B}_{{r_0},n})$ is a weak solution to the homogeneous boundary value problem  
\begin{equation}\label{eq:CPv_n}
\left\{\begin{aligned}
-\Delta v_n-f v_n&=fg_n+\Delta g_n &&\text{in}\ \tilde{B}_{{r_0},n}, \\
v_n&=0 && \text{on}\ \partial \tilde{B}_{{r_0},n},
\end{aligned}\right.
\end{equation}
i.e. 
\begin{equation*}\label{eq:CPvndeb}
\left\{\begin{aligned}
&v_n\in H^1_{0}(\tilde{B}_{{r_0},n}), \\
&\int_{\tilde{B}_{{r_0},n}} (\nabla v_n \cdot \nabla\phi-f v_n\phi) \,dx=\int_{\tilde{B}_{{r_0},n}} (f g_n+ \Delta g_n) \phi\,dx \quad\text{for any}\ \phi\in H^1_{0}(\tilde{B}_{{r_0},n}).
\end{aligned}\right. 
\end{equation*}
\begin{lem}\label{vnunica}
Let $r_0$ be as in Lemma \ref{lemmacoerc}. Then, for all $n\in\mathbb{N}$, problem \eqref{eq:CPv_n} has one and only one weak solution $v_n\in H^1_0(\tilde{B}_{{r_0},n})$, where $\tilde{B}_{{r_0},n}$ is defined in \eqref{Btilde}.
\end{lem}
\begin{proof}
Let us consider the bilinear form
\[
a(v,w)= \int_{\tilde{B}_{{r_0},n}} (\nabla v\cdot \nabla w-fvw) \,dx,
\]
for every $v,w\in H^1_0(\tilde{B}_{{r_0},n})$. Lemma \ref{lemmacoerc} implies that $a$ is coercive on $H^1_0(\tilde{B}_{{r_0},n})$. Furthermore, from estimate \eqref{equazcoercriscritta} we easily deduce that $a$ is continuous. The thesis then follows from the Lax-Milgram Theorem. 
\end{proof} 
\begin{prop}\label{prop:vnbound}
Under the same assumptions of Lemma \ref{vnunica}, there exists a positive constant $C>0$ such that $\Vert v_n\Vert_{H^1_0(B_{r_0})}\leq C$ for every $n\in\mathbb{N}$, where $v_n$ is extended trivially to zero in $B_{r_0}\setminus \tilde{B}_{{r_0},n}$.
\end{prop}
\begin{proof}
Let us observe that $fg_n$ and $-\Delta g_n$ are bounded in $H^{-1}(B_{r_0})$ as a consequence of the boundedness of $g_n$ in $H^1(B_{r_0})$: indeed, using \eqref{equazcoercriscritta}, one has that, for any $\phi\in H^1_0(B_{r_0})$,
\begin{equation}\label{fgnlim}
\begin{split}
\biggl|\int_{B_{r_0}}fg_n\phi\,dx\biggr|&\leq \biggl(\int_{B_{r_0}}|f|g_n^2\,dx\biggr)^{\frac{1}{2}}\biggl(\int_{B_{r_0}}|f|\phi^2\,dx\biggr)^{\frac{1}{2}}\\
&\leq \frac{1}{2}\biggl(\frac{1}{2}\int_{B_{r_0}}|\nabla g_n|^2\,dx+\omega(r_0)\int_{\partial B_{r_0}}g_n^2\,ds\biggr)^\frac{1}{2}\biggl(\int_{B_{r_0}}|\nabla \phi|^2\,dx\biggr)^\frac{1}{2}\\
&\leq c_1\Vert g_n\Vert_{H^1(B_{r_0})}\Vert \phi\Vert_{H^1_0(B_{r_0})},
\end{split}
\end{equation}
for some $c_1>0$ independent on $n$ and $\phi$, and
\begin{equation}\label{laplgnlim}
\biggl|-\int_{B_{r_0}}\Delta g_n\phi\,dx\biggr|=\biggl|\int_{B_{r_0}}\nabla g_n\cdot\nabla\phi\,dx\biggr|\leq c_2\Vert g_n\Vert_{H^1(B_{r_0})}\Vert \phi\Vert_{H^1_0(B_{r_0})},
\end{equation}
for some $c_2>0$ independent on $n$ and $\phi$. Thus from \eqref{eq:CPv_n}, \eqref{fgnlim}, \eqref{laplgnlim} and Lemma \ref{lemmacoerc}, it follows that
\begin{equation*}
\begin{split}
\Vert v_n\Vert^2_{H^1_0(B_{r_0})}&=\int_{B_{r_0}}|\nabla v_n|^2\,dx\leq 2\int_{B_{r_0}}(|\nabla v_n|^2-fv_n^2)\,dx=2\int_{B_{r_0}} (fg_n+\Delta g_n)v_n\,dx\\
&\leq 2(c_1+c_2)\Vert g_n\Vert_{H^1(B_{r_0})}\Vert v_n\Vert_{H^1_0(B_{r_0})}\leq c_3 \Vert v_n\Vert_{H^1_0(B_{r_0})},
\end{split}
\end{equation*}
for some $c_3>0$ independent on $n$. This completes the proof.
\end{proof}

\begin{prop}\label{unconvdebu}
Under the same assumptions of Lemma \ref{vnunica}, we have that $u_n\rightharpoonup u$ weakly in $H^1(B_{r_0})$, where $u_n$ is extended trivially to zero in $B_{r_0}\setminus \tilde{B}_{r_0,n}$. 
\end{prop}
\begin{proof}
 We observe that the trivial extension to zero of $u_n$ in
$B_{r_0}\setminus \tilde{B}_{r_0,n}$ belongs to $H^1(B_{r_0})$ since
the trace of $u_n$ on $\tilde{\gamma}_{r_0,n}$ is null in view of
Remark \ref{rem:g_n_zero}. 

From Proposition \ref{prop:vnbound} it follows that there exist
$\tilde{v}\in H^1_0(B_{r_0})$ and a subsequence $\{v_{n_k}\}$ of
$\{v_n\}$ such that $v_{n_k}\rightharpoonup \tilde{v}$ weakly in
$H^1_0(B_{r_0})$. Then
$u_{n_k}=v_{n_k}+g_{n_k}\rightharpoonup \tilde{u}$ weakly in $H^1(B_{r_0})$,
where $\tilde{u}:=\tilde{v}+u$.  Let
$\phi\in C^\infty_c(B_{r_0}\setminus \Gamma)$. Arguing as in
Remark \ref{rem:g_n_zero}, we can prove that 
$\phi\in H^1_0(\tilde{B}_{r_0,n_k})$ for all sufficiently large
$k$. Hence, from \eqref{eq:CPu_n} it follows that, for all sufficiently large
$k$,
\begin{equation}\label{unk}
\int_{B_{r_0}}\nabla u_{n_k}\cdot\nabla \phi \,dx=\int_{B_{r_0}}fu_{n_k}\phi \,dx,
\end{equation}
where $u_{n_k}$ is extended trivially to zero in
$B_{r_0}\setminus \tilde{B}_{r_0,n_k}$. Passing to the limit in
\eqref{unk}, we obtain that
\begin{equation*}
\int_{B_{r_0}}\nabla \tilde{u}\cdot\nabla\phi\,dx=\int_{B_{r_0}}f \tilde{u}\phi\,dx
\end{equation*}
for every $\phi\in C^\infty_c(B_{r_0}\setminus \Gamma)$.  Furthermore
$\tilde{u}=u$ on $\partial B_{r_0}$ in the trace sense: indeed, due to
compactness of the trace map  $\gamma:H^1(B_{r_0})\to  L^2(\partial B_{r_0})$, we have that
$\gamma(u_{n_k})\rightarrow\gamma(\tilde{u})$ in $L^2(\partial
B_{r_0})$ and $\gamma(u_{n_k})=\gamma(g_{n_k})\rightarrow \gamma(u)$ in $L^2(\partial B_{r_0})$,
since $g_n\rightarrow u\ \text{in}\ H^1(B_{r_0})$. 

 Finally, we prove
that $\tilde{u}\in H^1_\Gamma(B_{r_0})$.
To this aim, let $\Gamma_\delta=  
\{(x',x_N)\in \mathbb{R}^N:\, x_N\geq g(x')+\delta\}$  for every $\delta>0$. For every
$\delta>0$ we have that $\Gamma_\delta \cap B_{r_0}\subset
B_{r_0}\setminus \tilde{B}_{r_0,n}$ provided $n$ is sufficiently
large. Hence, since $u_{n}$ is extended trivially to zero in
$B_{r_0}\setminus \tilde{B}_{r_0,n}$, we have that, for every $\delta>0$, $u_n\in
H^1_{\Gamma_\delta}(B_{r_0})$ provided $n$ is sufficiently
large. Since $H^1_{\Gamma_\delta}(B_{r_0})$ is weakly
closed in $H^1(B_{r_0})$, it follows that $\tilde{u}\in
H^1_{\Gamma_\delta}(B_{r_0})$ for every $\delta>0$, and hence
$\tilde{u}\in H^1_\Gamma(B_{r_0})$. 
 
Thus $\tilde{u}$ weakly solves
\begin{equation*}\label{eq:CPutilde}
\left\{\begin{aligned}
-\Delta \tilde{u}&=f \tilde{u} &&\text{in}\ B_{r_0}\setminus\Gamma, \\
\tilde{u}&=u && \text{on}\ \partial B_{r_0}, \\
\tilde{u}&=0 && \text{on}\ \Gamma.
\end{aligned}\right. 
\end{equation*}
Now we consider the function $U:=\tilde{u}-u$: it weakly solves the following problem
\begin{equation}\label{eq:CPU}
\left\{\begin{aligned}
-\Delta U&=f U &&\text{in}\ B_{r_0}\setminus\Gamma, \\
U&=0 && \text{on}\ \partial B_{r_0}, \\
U&=0 && \text{on}\ \Gamma.
\end{aligned}\right. 
\end{equation}
Testing equation \eqref{eq:CPU} with $U$ itself and using Lemma \ref{lemmacoerc}, we obtain that 
\begin{equation*}
\frac{1}{2}\int_{B_{r_0}}|\nabla U|^2\,dx\leq \int_{B_{r_0}}(|\nabla U|^2-fU^2)\,dx=0,
\end{equation*}
so that $U=0$, hence $u=\tilde{u}$. By Urysohn's subsequence
principle, we can conclude that $u_n\rightharpoonup u$ weakly in $H^1(B_{r_0})$.   
\end{proof}

 Our next aim is to prove strong convergence of the sequence $\{u_n\}_{n\in\mathbb{N}}$ to $u$ in $H^1(B_{r_0})$.
\begin{prop}\label{unconvu}
  Under the same assumptions of Lemma \ref{vnunica}, we have that
  $u_n\rightarrow u$ in $H^1(B_{r_0})$.
\end{prop}
\begin{proof}
  From Proposition \ref{unconvdebu} we deduce that
  $v_n\rightharpoonup 0$ in $H^1(B_{r_0})$, hence testing
  \eqref{eq:CPv_n} with $v_n$ itself, we have that
\begin{equation*}
\begin{split}
\int_{B_{r_0}}(|\nabla v_n|^2-fv_n^2)\,dx&=\int_{\tilde{B}_{r_0,n}}(|\nabla v_n|^2-fv_n^2)\,dx\\
&=\int_{\tilde{B}_{r_0,n}}(fg_nv_n-\nabla g_n\nabla v_n)\,dx=\int_{B_{r_0}}(fg_nv_n-\nabla g_n\nabla v_n)\,dx\rightarrow 0
\end{split}
\end{equation*}
as\ $n\rightarrow\infty$.
Thus, from Lemma \ref{lemmacoerc}, we deduce that $\Vert v_n\Vert_{H^1_0(B_{r_0})}\rightarrow 0$ as $n\rightarrow\infty$, hence $v_n\rightarrow 0$ in $H^1(B_{r_0})$. This yields that $u_n=g_n+v_n\rightarrow u$ in $H^1(B_{r_0})$.
\end{proof}
\section{Pohozaev Identity}\label{sec:pohozaev-identity}
 In this section we derive a Pohozaev type identity for $u_n$ 
   in which we will pass to the limit using Proposition \ref{unconvu}.
For every $r\in(0,r_0)$ and $v\in H^1(B_r)$, we define 
\begin{equation*}
\mathcal R(r,v)=
\begin{cases}
{\displaystyle\int_{B_r}} f v (x\cdot \nabla v)\,dx, &\text{if $f$ satisfies 
\eqref{xi0}-\eqref{xi}},\\[10pt]
\dfrac{r}{2}{\displaystyle\int_{\partial B_r}}f\,v^2\,dS-\dfrac{1}{2}{\displaystyle\int_{B_r}}\big(\nabla
f\cdot x+(N+1)f\big) v^2\,dx, &\text{if $f$ satisfies \eqref{eta0}-\eqref{eta}}. 
\end{cases}
\end{equation*}

\begin{lem}\label{LemmaPohoz} Let $0<r< r_0$. 
 There exists $n_0=n_0(r)\in\mathbb{N}\setminus\{0\}$ such that, for  all $n\geq
n_0$, 
\begin{multline}\label{pohozaevun}
-\frac{N-1}{2}\int_{\tilde{B}_{r,n}}|\nabla u_n|^2\,dx+\frac{r}{2}\int_{\tilde{S}_{r,n}}|\nabla u_n|^2\,dS\\-\frac{1}{2}\int_{\tilde{\gamma}_{r,n}}\biggl|\frac{\partial u_n}{\partial\nu}\biggr|^2x\cdot\nu\,dS
-r\int_{\tilde{S}_{r,n}} \biggl|\frac{\partial u_n}{\partial
  \nu}\biggr|^2\,dS-\mathcal R(r,u_n) = 0.
\end{multline}
\end{lem}
\begin{proof}
  Since $u_n$ solves \eqref{eq:CPu_n} in the domain
  $\tilde{B}_{r_0,n}$, which satisfies the exterior ball condition, and
  $fu_n\in L^2_{\text{loc}}(\tilde{B}_{r_0,n}\setminus\{0\})$, by
  elliptic regularity theory (see \cite{Adolf}) we have that
  $u_n\in H^2(\tilde{B}_{r,n}\setminus B_\delta)$ for all
  $r\in (0,r_0)$, $n$ sufficiently large and all $\delta<r_n$, where
  $r_n$ is such that $B_{r_n}\subset \tilde{B}_{r,n}$. Since
\begin{equation*}
\int_0^{r_n}\biggl[\int_{\partial B_r}\big(|\nabla u_n|^2+|f|u_n^2\big)\,dS\biggr]\, dr=\int_{B_{r_n}}
\big(|\nabla u_n|^2 +|f|u_n^2\big)\,dx<+\infty,
\end{equation*}
there exists a sequence $\{\delta_k\}_{k\in\mathbb{N}}\subset (0,r_n)$ such that $\lim_{k\rightarrow\infty}\delta_k=0$ and 
\begin{equation}\label{deltak}
\delta_k\int_{\partial B_{\delta_k}}|\nabla u_n|^2\,dS\rightarrow 0,
\quad 
\delta_k\int_{\partial B_{\delta_k}}|f|u_n^2\,dS\rightarrow 0
\quad \text{as $k\rightarrow\infty$}.
\end{equation}
Testing \eqref{eq:CPu_n} with $x\cdot\nabla u_n$ and integrating over $\tilde{B}_{r,n}\setminus B_{\delta_k}$, we obtain that
\begin{equation}\label{xnablaun1}
-\int_{\tilde{B}_{r,n}\setminus B_{\delta_k}}\Delta u_n(x\cdot\nabla u_n)\,dx=\int_{\tilde{B}_{r,n}\setminus B_{\delta_k}}fu_n(x\cdot\nabla u_n)\,dx.
\end{equation}
 Integration by parts allows us to rewrite the first term in
\eqref{xnablaun1} as 
\begin{equation}\label{xnablaun2}
\begin{split}
-\int_{\tilde{B}_{r,n}\setminus B_{\delta_k}}\Delta u_n(x\cdot\nabla u_n)\,dx=&\int_{\tilde{B}_{r,n}\setminus B_\delta}\nabla u_n\cdot\nabla(x\cdot \nabla u_n)\,dx-r\int_{\tilde{S}_{r,n}}\biggl|\frac{\partial u_n}{\partial\nu}\biggr|^2\,dS\\
&-\int_{\tilde{\gamma}_{r,n}}\biggl|\frac{\partial u_n}{\partial\nu}\biggr|^2 x\cdot\nu\,dS+\delta_k\int_{\partial B_{\delta_k}}\biggl|\frac{\partial u_n}{\partial\nu}\biggr|^2\,dS,
\end{split}
\end{equation}
where we used that $x=r\nu$ on $\tilde{S}_{r,n}$ and that the
tangential component of $\nabla u_n$ on $\tilde{\gamma}_{r,n}$ equals
zero, thus $\nabla u_n=\frac{\partial u_n}{\partial \nu}\nu$ on
$\tilde{\gamma}_{r,n}$. Furthermore, by direct calculations, the first term in \eqref{xnablaun2} can be rewritten as 
\begin{equation}\label{xnablaun3}
\begin{split}
\int_{\tilde{B}_{r,n}\setminus B_{\delta_k}}\nabla u_n\cdot\nabla(x\cdot \nabla u_n)\,dx=&-\frac{N-1}{2}\int_{\tilde{B}_{r,n}\setminus B_{\delta_k}}|\nabla u_n|^2\,dx+\frac{r}{2}\int_{\tilde{S}_{r,n}}|\nabla u_n|^2\,dS\\
&+\frac{1}{2}\int_{\tilde{\gamma}_{r,n}}\biggl|\frac{\partial u_n}{\partial\nu}\biggr|^2 x\cdot\nu\,dS-\frac{\delta_k}{2}\int_{\partial B_{\delta_k}}|\nabla u_n|^2\,dS.
\end{split}
\end{equation}
Taking into account \eqref{xnablaun1}, \eqref{xnablaun2} and \eqref{xnablaun3}, we obtain that
\begin{multline}\label{eq:13}
  -\frac{N-1}{2}\int_{\tilde{B}_{r,n}\setminus B_{\delta_k}}|\nabla
  u_n|^2\,dx+\frac{r}{2}\int_{\tilde{S}_{r,n}}|\nabla
  u_n|^2\,dS-\frac{1}{2}\int_{\tilde{\gamma}_{r,n}}\biggl|\frac{\partial
    u_n}{\partial \nu}\biggr|^2 x\cdot\nu\,dS
  -r\int_{\tilde{S}_{r,n}}\biggl|\frac{\partial u_n}{\partial
    \nu}\biggr|^2\,dS\\
  -\frac{\delta_k}{2}\int_{\partial B_{\delta_k}}|\nabla u_n|^2\,dS
  +\delta_k\int_{\partial B_{\delta_k}}\biggl|\frac{\partial
    u_n}{\partial\nu}\biggr|^2\,dS-\int_{\tilde{B}_{r,n}\setminus
    B_{\delta_k}}fu_n(x\cdot\nabla u_n)\,dx=0.
\end{multline}
Under assumptions \eqref{xi0}-\eqref{xi}, the Hardy inequality
\eqref{Hardy} implies that $f\,u_n(x\cdot\nabla u_n)\in
L^1(B_r)$ and hence 
\begin{equation}\label{eq:14}
\lim_{k\to\infty}
\int_{\tilde{B}_{r,n}\setminus
    B_{\delta_k}}fu_n(x\cdot\nabla
  u_n)\,dx=
\lim_{k\to\infty}
\int_{B_r\setminus
    B_{\delta_k}}fu_n(x\cdot\nabla
  u_n)\,dx=\int_{B_r}fu_n(x\cdot\nabla u_n)\,dx.
\end{equation}
On the other hand, if \eqref{eta0}-\eqref{eta} hold, 
 we can use the Divergence Theorem to obtain that
\begin{multline}\label{eq:9}
  \int_{\tilde{B}_{r,n}\setminus
    B_{\delta_k}}fu_n(x\cdot\nabla u_n)\,dx\\=
\frac{r}{2}\int_{\tilde{S}_{r,n}}fu_n^2\,dS-\frac{1}{2}\int_{\tilde{B}_{r,n}\setminus
    B_{\delta_k}}\big(\nabla f\cdot x+(N+1)f\big)u_n^2\,dx-
\frac{\delta_k}{2}\int_{\partial B_{\delta_k}}fu_n^2\,dS\\=
\frac{r}{2}\int_{\partial B_r}fu_n^2\,dS-\frac{1}{2}\int_{B_r\setminus
    B_{\delta_k}}\big(\nabla f\cdot x+(N+1)f\big)u_n^2\,dx-
\frac{\delta_k}{2}\int_{\partial B_{\delta_k}}fu_n^2\,dS. 
\end{multline}
Since, under assumptions \eqref{eta0}-\eqref{eta}, $\big(\nabla f\cdot x+(N+1)f\big)u_n^2 \in
L^1(B_r)$, we can pass to the limit as $k\to\infty$ in
\eqref{eq:9} taking into account also \eqref{deltak}, thus obtaining
that 
\begin{equation}\label{eq:16}
\lim_{k\to\infty}
\int_{\tilde{B}_{r,n}\setminus
    B_{\delta_k}}fu_n(x\cdot\nabla
  u_n)\,dx=\frac{r}{2}\int_{\partial B_r}fu_n^2\,dS-\frac{1}{2}\int_{B_r}
\big(\nabla f\cdot x+(N+1)f\big)u_n^2\,dx.
\end{equation}
Letting $k\rightarrow +\infty$ in \eqref{eq:13}, by \eqref{deltak},
\eqref{eq:14}, and \eqref{eq:16}, we obtain
\eqref{pohozaevun}. 
\end{proof}
Combining Lemma \ref{LemmaPohoz} with the fact that the 
 domains $\tilde{B}_{r,n}$ (defined as in \eqref{Btilde}) are
 star-shaped with respect to the origin, we deduce  the following inequality. 
\begin{cor} 
Let $0<r< r_0$. 
 There exists $n_0=n_0(r)\in\mathbb{N}\setminus\{0\}$ such that, for  all $n\geq
n_0$, 
\begin{equation}\label{pohozinequality}
  -\frac{N-1}{2}\int_{\tilde{B}_{r,n}}|\nabla
  u_n|^2\,dx+\frac{r}{2}\int_{\tilde{S}_{r,n}}|\nabla
  u_n|^2\,dS-r\int_{\tilde{S}_{r,n}}\biggl|\frac{\partial
    u_n}{\partial \nu}\biggr|^2\,dS-
\mathcal R(r,u_n)\geq 0. 
\end{equation}
\end{cor}
\begin{proof}
 In view of \eqref{pohozaevun}, the left-hand side of
\eqref{pohozinequality} is equal to
$\frac{1}{2}\int_{\tilde{\gamma}_{r,n}}\big|\frac{\partial
  u_n}{\partial \nu}\big|^2 x\cdot\nu\,dS$, which is in fact
non-negative, 
since $x\cdot\nu\geq 0$ on $\tilde{\gamma}_{r,n}$ by Lemma \ref{l:stellato}. 
\end{proof}
Passing to the limit in \eqref{pohozinequality}  as $n\to\infty$, 
a similar inequality can be derived for $u$.
\begin{prop}
Let $u$ solve \eqref{eq:CPu},   with $f$ satisfying
  either \eqref{xi0}-\eqref{xi} or
  \eqref{eta0}-\eqref{eta}.  Then, for a.e. $r\in (0,r_0)$, we have that
\begin{equation}\label{pohoz}
  -\frac{N-1}{2}\int_{B_r}|\nabla u|^2\,dx+\frac{r}{2}\int_{\partial
    B_r} |\nabla u|^2 \,dS- r\int_{\partial B_r} \biggl|\frac{\partial
    u}{\partial \nu}\biggr|^2 \,dS-\mathcal R(r,u)\geq 0
\end{equation}
and
\begin{equation}\label{intgrad}
\int_{B_r}|\nabla u|^2\,dx=\int_{B_r} fu^2\,dx+ \int_{\partial B_r}u\frac{\partial u}{\partial\nu}\,dS.
\end{equation}
\end{prop}
\begin{proof}
In order to prove \eqref{pohoz}, we pass to the limit inside inequality \eqref{pohozinequality}. As regards the first term, it is sufficient to observe that  
\begin{equation*}
\int_{\tilde{B}_{r,n}}|\nabla u_n|^2 dx=\int_{B_r}|\nabla u_n|^2 dx\rightarrow\int_{B_r}|\nabla u|^2 dx\quad\text{as $n\rightarrow\infty$},
\end{equation*} 
for each fixed $r\in (0,r_0)$, as a consequence of Proposition
\ref{unconvu}. In order to deal with the second term, we observe that,
by strong $H^1$-convergence of $u_n$ to $u$,
\begin{equation}\label{Bt}
\lim_{n\rightarrow +\infty}\int_0^{r_0}\biggl(\int_{\partial B_r}|\nabla(u_n-u)|^2\,dS\biggr)\,dr= 0.
\end{equation}
Letting
\[
F_n(r)=\int_{\partial B_r}|\nabla(u_n-u)|^2\,dS,
\] 
\eqref{Bt} implies that $F_n\rightarrow 0$ in $L^1(0,r_0)$. Then there exists a subsequence $F_{n_k}$ such that $F_{n_k}(r)\rightarrow 0$ for a.e. $r\in (0,r_0)$, hence 
\[
\int_{\tilde{S}_{r,n_k}}|\nabla u_{n_k}|^2\,dS=\int_{\partial B_r}|\nabla u_{n_k}|^2\,dS \rightarrow \int_{\partial B_r}|\nabla u|^2\,dS\quad\text{as $k\rightarrow\infty$}\]
for a.e. $r\in (0,r_0)$. In a similar way, we obtain that 
\[
\int_{\tilde{S}_{r,n_k}}\biggl|\frac{\partial u_{n_k}}{\partial\nu}\biggr|^2 dS\rightarrow\int_{\partial B_r}\biggl|\frac{\partial u}{\partial\nu}\biggr|^2 dS\quad\text{as $k\rightarrow\infty$}.
\]
It remains to prove the convergence of $\mathcal R(r,u_n)$. Under the set of assumptions
\eqref{xi0}-\eqref{xi}, we first write
\begin{equation}\label{eq:7}
  \begin{split}
    \int_{B_r}|fu_n(x\cdot \nabla u_n)&-fu(x\cdot \nabla
    u)|dx=\int_{B_r} |f(u_n-u)(x\cdot \nabla u_n)-fux\cdot\nabla
    (u-u_n)| dx\\
& \leq \int_{B_r} |f(u_n-u)(x\cdot \nabla u_n)| dx +
\int_{B_r} |fux\cdot\nabla
    (u-u_n)| dx.
  \end{split}
\end{equation}
The H\"{o}lder inequality, \eqref{Hardy}, and Proposition \ref{unconvu}
imply that
\begin{multline*}
\int_{B_r}|f(u_n-u)(x\cdot\nabla u_n) |dx \leq \xi_f(r)\biggl(\int_{B_r}\frac{|u_n-u|^2}{|x|^2}dx\biggr)^{1/2}\biggl(\int_{B_r} |\nabla u_n|^2 dx\biggr)^{1/2}\\
\leq \frac{2}{N-1}\xi_f(r)\biggl(\int_{B_r}|\nabla (u_n-u)|^2 dx+ \frac{N-1}{2r}\int_{\partial B_r}|u_n-u|^2 dS\biggr)^{1/2}\biggl(\int_{B_r} |\nabla u_n|^2 dx\biggr)^{1/2}\rightarrow 0
\end{multline*}
and
\begin{multline*}
\int_{B_r}|fux\cdot\nabla (u_n-u)|dx \leq \xi_f(r)\biggl(\int_{B_r}\frac{|u(x)|^2}{|x|^2}dx\biggr)^{1/2}\biggl(\int_{B_r} |\nabla (u_n-u)|^2 dx\biggr)^{1/2}\\
\leq \frac{2}{N-1}\xi_f(r)\biggl(\int_{B_r}|\nabla u|^2 dx+ \frac{N-1}{2r}\int_{\partial B_r}|u|^2 dS\biggr)^{1/2}\biggl(\int_{B_r} |\nabla (u_n-u)|^2 dx\biggr)^{1/2}\rightarrow 0
\end{multline*}
as $n\rightarrow\infty$, for a.e. $r\in (0,r_0)$, since $\xi_f(r)$ is
finite a.e. as a consequence of assumption \eqref{xiL1}. 
Hence, from \eqref{eq:7} we deduce that
\begin{equation}\label{eq:8}
\lim_{n\to\infty}\mathcal R(r,u_n)=\mathcal R(r,u)
  \end{equation}
under assumptions \eqref{xi0}-\eqref{xi}.  To prove \eqref{eq:8} 
under assumptions \eqref{eta0}-\eqref{eta},
we first use Proposition \ref{unconvu} and the H\"{o}lder inequality
to observe that 
\begin{equation*}
\begin{split}
&\bigg|\int_{B_r}[\nabla f\cdot x+(N+1)f](u_n^2-u^2)\,dx \bigg|\\
&\leq
\biggl(\int_{B_r}(
|\nabla f\cdot
x|+(N+1)|f|)|u_n-u|^2\,dx\biggr)^{\frac{1}{2}}\biggl(\int_{B_r}
(
|\nabla f\cdot x|+(N+1)|f|)|u_n+u|^2\,dx\biggr)^{\frac{1}{2}}\\
&\leq (\eta(r,\nabla f\cdot x)+
(N+1) \eta(r,f))
\biggl(\int_{B_r}|\nabla(u_n-u)|^2\,dx+\tfrac{N-1}{2r}\int_{\partial B_r}|u_n-u|^2\,dS\biggr)^{\frac{1}{2}}\\
&\hskip6cm\cdot\biggl(\int_{B_r}|\nabla(u_n+u)|^2\,dx+\tfrac{N-1}{2r}\int_{\partial B_r}|u_n+u|^2\,dS\biggr)^{\frac{1}{2}}\rightarrow 0,
\end{split}
\end{equation*}
as $n\rightarrow\infty$, for a.e. $r\in (0,r_0)$, since
$\eta(r,\nabla f\cdot x)$ and $\eta(r,f) $ are finite a.e. as a
consequence of assumptions \eqref{etanablafL1} and \eqref{etaL1} and
$\{u_n+u\}_n$ is bounded in $H^1(B_r)$ for every $r\in
(0,r_0)$.
Furthermore, by the fact that $f$ is bounded far from the origin and
the compactness of the trace map from $H^1(B_r)$ to
$L^2(\partial B_r)$, it follows that
\begin{equation*}
\int_{\partial B_r}fu_n^2\,dS\rightarrow \int_{\partial B_r}fu^2\,dS,
\end{equation*} 
for a.e. $r\in (0,r_0)$.
 Hence, passing  to the limit in $\mathcal R(r,u_n)$ we conclude the first part of the proof. 

 Finally \eqref{intgrad}
follows by testing \eqref{eq:CPu_n} with $u_n$ itself and passing
to the limit arguing as above.
\end{proof}
\section{The Almgren type frequency function}\label{sec:almgr-type-freq}
Let $u\in H^1_\Gamma(B_{\hat{R}})$ be a non trivial solution to \eqref{eq:CPu}. For every $r\in (0,\hat{R})$ we define 
\begin{equation}\label{D}
\mathcal{D}(r)=r^{1-N}\int_{B_r}\big(|\nabla u|^2-fu^2\big)\,dx
\end{equation}
and 
\begin{equation}\label{H}
\mathcal{H}(r)=r^{-N}\int_{\partial B_r} u^2 \,dS.
\end{equation}
In the following lemma we compute the derivative of the function $\mathcal{H}$.
\begin{lem}\label{LemH'}
We have that $\mathcal{H}\in W^{1,1}_{\mathrm{loc}}(0,\hat{R})$ and
\begin{equation}\label{H'}
\mathcal{H}'(r)=2r^{-N}\int_{\partial B_r}u\frac{\partial u}{\partial\nu}\,dS
\end{equation}
in a distributional sense and for a.e. $r\in (0,\hat{R})$. Furthermore 
\begin{equation}\label{H'2}
\mathcal{H}'(r)=\frac{2}{r}\mathcal{D}(r)\quad \text{for a.e. }r\in (0,\hat{R}).
\end{equation}
\end{lem}
\begin{proof}
First we observe that 
\begin{equation}\label{Hcambiodivar}
\mathcal{H}(r)= \int_{\mathbb{S}^N} |u(r\theta)|^2 \,dS.
\end{equation}
Let $\phi \in C^\infty_c(0,\hat{R})$. Since $u, \nabla u\in L^2(B_{\hat{R}})$, we obtain that
\begin{align*}
-\int_0^{\hat{R}} \mathcal{H}(r)\phi'(r)\,dr&=-\int_0^{\hat{R}}\biggl(\int_{\partial B_1} u^2(r\theta) \,dS\biggr)\phi'(r)\,dr\\
&=-\int_{B_{\hat{R}}} |x|^{-N-1} u^2(x)\nabla v(x)\cdot x \,dx=2\int_{B_{\hat{R}}} v(x)|x|^{-N-1}u\nabla u\cdot x\,dx\\
&=2\int_0^{\hat{R}} \phi(r)\biggl(\int_{\partial B_1}u (r\theta)\nabla u(r\theta)\cdot \theta\,dS\biggr)\,dr,
\end{align*}
where we set $v(x)=\phi(|x|)$. Thus we proved \eqref{H'}. Identity \eqref{H'2} follows from \eqref{H'} and \eqref{intgrad}.
\end{proof}
{We now observe that the function $\mathcal H$ is strictly
  positive in a neighbourhood of $0$.
\begin{lem}\label{H>0}
For any $r\in (0,r_0]$, we have that $\mathcal{H}(r)>0$. 
\end{lem}
\begin{proof}
  Assume by contradiction that there exists $r_1\in (0,r_0]$ such that
  $\mathcal{H}(r_1)=0$, so that  the trace of $u$ on $\partial
  B_{r_1}$ is null and hence $u\in
  H^1_0(B_{r_1}\setminus\Gamma)$. Then, testing \eqref{eq:CPu} with
  $u$, we obtain that 
\begin{equation}\label{barr}
\int_{B_{r_1}}|\nabla u|^2\,dx-\int_{B_{r_1}}fu^2\,dx=0.
\end{equation}
Thus, from Lemma \ref{lemmacoerc} and \eqref{barr} it follows  that 
\begin{equation*}
0=\int_{B_{r_1}}[|\nabla u|^2-fu^2]\,dx\geq \frac{1}{2}\int_{B_{r_1}}|\nabla u|^2\,dx,
\end{equation*}
which, together with Lemma \ref{lemmaHardy}, implies that $u\equiv 0$ in $B_{r_1}$. From classical unique continuation principles for second order elliptic equations with locally bounded coefficients (see e.g. \cite{Wolff}), we can conclude that $u=0$ a.e. in $B_{\hat{R}}$, a contradiction. 
\end{proof}

 Let us now  differentiate the function $\mathcal{D}$ and estimate
 from below its derivative.
\begin{lem} \label{LemD'}
The function $\mathcal{D}$ defined in \eqref{D} belongs to
$W^{1,1}_{\mathrm{loc}}(0,\hat{R})$ and 
\begin{equation}\label{D'1}
\mathcal{D}'(r)\geq 2r^{1-N} \int_{\partial B_r} \biggl|\frac{\partial u}{\partial \nu}\biggr|^2 \,dS+(N-1)r^{-N}\int_{B_r}fu^2\,dx+2r^{-N}\mathcal R(r,u)- r^{1-N}\int_{\partial B_r} fu^2\, dS
\end{equation}
for a.e. $r\in (0,r_0)$.
\end{lem}
\begin{proof}
We have that 
\begin{equation}\label{D'}
\mathcal{D}'(r)= (1-N)r^{-N}\int_{B_r}\big(|\nabla
u|^2-fu^2\big)\,dx+r^{1-N} \int_{\partial B_r}\big(
|\nabla u|^2-fu^2\big)\,dS
\end{equation}
for a.e. $r\in (0,r_0)$ and in the distributional sense. Combining \eqref{pohoz} and \eqref{D'}, we obtain \eqref{D'1}.
\end{proof}
Thanks to Lemma \ref{H>0}, the frequency function 
\begin{equation}\label{frequency}
\mathcal{N}\colon (0,r_0 ]\to \mathbb{R},\quad\mathcal{N}(r)=\frac{\mathcal{D}(r)}{\mathcal{H}(r)}
\end{equation}
is well defined.
Using
Lemmas \ref{LemH'}, \ref{LemD'}, 
 and \ref{lemmacoerc} 
we can estimate from below $\mathcal N$ and its derivative. 
\begin{lem} \label{LemN'}
The function $\mathcal{N}$ defined in \eqref{frequency} belongs to $W^{1,1}_{\mathrm{loc}}((0,r_0])$ and
\begin{equation}\label{eq:10}
\mathcal{N}'(r)\geq\nu_1(r)+\nu_2(r),
\end{equation}
for a.e. $r\in (0,r_0)$, where
\begin{equation*}
\nu_1(r)=\frac{2r\bigl[\bigl(\int_{\partial B_r}\bigl|\frac{\partial u}{\partial \nu}\bigr|^2 \,dS\bigr)\bigl(\int_{\partial B_r}|u|^2\,dS\bigr)-\bigl(\int_{\partial B_r}u\frac{\partial u}{\partial\nu}\,dS\bigr)^2\bigr]}{\bigl(\int_{\partial B_r}|u|^2\,dS\bigr)^2}
\end{equation*}
and
\begin{equation}\label{nu2}
\nu_2(r)=\frac{2\bigl[\frac{N-1}{2}\int_{B_r}fu^2\,dx+\mathcal R(r,u)-\frac{r}{2}\int_{\partial B_r}fu^2\,dS\bigr]}{\int_{\partial B_r}|u|^2\,dS}.
\end{equation}
 Furthermore, 
\begin{equation}\label{Nlimitatafrombelow}
\mathcal N(r)>-\frac{N-1}4\quad\text{for every }r\in(0,r_0)
\end{equation}
and, 
for every $\varepsilon>0$, there exists $r_\varepsilon>0$
such that  
\begin{equation}\label{eq:15}
\mathcal{N}(r)>-\varepsilon\quad\text{for every $r\in (0,r_\varepsilon)$},
\end{equation}
i.e. $\liminf_{r\to 0^+}\mathcal N(r)\geq0$.
\end{lem}
\begin{proof}
From Lemmas \ref{LemH'}, \ref{H>0}, and \ref{LemD'}, it follows that
$\mathcal{N}\in W^{1,1}_{\mathrm{loc}}((0,r_0])$. From \eqref{H'2} we deduce that
\begin{equation*}
\mathcal{N}'(r)=\frac{\mathcal{D}'(r)\mathcal{H}(r)-\mathcal{D}(r)\mathcal{H}'(r)}{(\mathcal{H}(r))^2}=\frac{\mathcal{D}'(r)\mathcal{H}(r)-\frac{1}{2}r(\mathcal{H}'(r))^2}{(\mathcal{H}(r))^2}
\end{equation*}
and the proof of \eqref{eq:10} easily follows from \eqref{H'} and
\eqref{D'1}. To prove \eqref{Nlimitatafrombelow} and \eqref{eq:15}, we observe that
\eqref{D} and \eqref{H}, together with Lemma \ref{lemmacoerc}, imply that
\begin{equation}\label{eq:12}
\begin{split}
\mathcal{N}(r)=\frac{\mathcal{D}(r)}{\mathcal{H}(r)}\geq \frac{r\bigl[\frac{1}{2}\int_{B_r}|\nabla u|^2 dx-\omega(r)\int_{\partial B_r} |u|^2\,dS\bigr]}{\int_{\partial B_r} |u|^2\,dS}\geq -r\omega(r)
\end{split}
\end{equation}
for every $r\in (0,r_0)$,  where $\omega$ is defined in
\eqref{eq:omega}. Then \eqref{Nlimitatafrombelow} follows directly
from \eqref{eq:bound-omega}. 
From either assumption  \eqref{xi0} or
\eqref{eta0} it follows that $\lim_{r\to 0^+}r\omega(r)=0$; hence
\eqref{eq:12} implies \eqref{eq:15}. 
\end{proof}
\begin{lem}\label{Lemstimanu2}
Let  $\nu_2$ be as in \eqref{nu2}. There exists a positive constant
$C_1>0$ 
 such that
\begin{equation}\label{estimatenu2}
|\nu_2(r)|\leq C_1\alpha(r)\biggl[\mathcal{N}(r)+\frac{N-1}{2}\biggr]
\end{equation}
for all $r\in (0,r_0)$, where 
\begin{equation}\label{eq:alpha}
 \alpha(r)=
\begin{cases}
r^{-1}\xi_f(r),&\text{under assumptions
  \eqref{xi0}-\eqref{xi}},\\
r^{-1}\Big(\eta(r,f)+\eta(r,\nabla f\cdot x)\Big), &\text{under assumptions \eqref{eta0}-\eqref{eta}}.
\end{cases}
\end{equation}
\end{lem}
\begin{proof}
From Lemma \ref{lemmacoerc} we deduce that, for all $r\in (0,r_0)$,
\begin{equation}\label{nablaeta}
\int_{B_r}|\nabla u|^2\,dx\leq 2\bigl(r^{N-1}\mathcal{D}(r)+\omega(r)r^{N}\mathcal{H}(r)\bigr),
\end{equation}
where $\omega(r)$ is defined in \eqref{eq:omega}.

Let us first suppose to be under assumptions \eqref{xi0}-\eqref{xi}. 
Estimating the first term in the numerator of $\nu_2(r)$ we obtain that
\begin{equation}\label{primoterm}
\begin{split}
\biggl|\int_{B_r}&fu^2\,dx\biggr|\leq\xi_f(r)\int_{B_r}\frac{|u(x)|^2}{|x|^2}\,dx\leq \xi_f(r)\frac{4}{(N-1)^2}\biggl[\int_{B_r}|\nabla u|^2\,dx+\frac{N-1}{2r}\int_{\partial B_r}u^2\,dS\biggr]\\
&\leq \frac{8}{(N-1)^2}r^{N-1}\xi_f(r)\mathcal{D}(r)+\frac{16}{(N-1)^3}r^{N-1}(\xi_f(r))^2\mathcal{H}(r)+\frac{2}{N-1}r^{N-1}\xi_f(r)\mathcal{H}(r)\\
&\leq
\frac{8}{(N-1)^2}r^{N-1}\xi_f(r)\mathcal{D}(r)+\frac{4}{N-1}r^{N-1}\xi_f(r)\mathcal{H}(r)\\&
=\frac{8}{(N-1)^2}r^{N-1}\xi_f(r)\left(\mathcal{D}(r)+\frac{N-1}2\mathcal{H}(r)\right),
\end{split}
\end{equation}
where we used \eqref{xi}, Lemma \ref{lemmaHardy},
\eqref{nablaeta} and \eqref{xir0}. 
 Using H\"{o}lder inequality, \eqref{primoterm}, \eqref{xir0}, and
\eqref{nablaeta}, the second term can be estimated as follows
\begin{equation}\label{fuxcdotnablau}
\begin{split}
\biggl|\int_{B_r}fux\cdot \nabla u\,dx&\biggr|\leq \xi_f(r)\biggl(\int_{B_r}\frac{|u(x)|^2}{|x|^2}\,dx\biggr)^\frac{1}{2}\biggl(\int_{B_r}|\nabla u|^2\,dx\biggr)^\frac{1}{2}\\
&\leq
\xi_f(r)\frac{4}{N-1}r^{N-1}\left(\mathcal{D}(r)+\frac{N-1}2\mathcal{H}(r)\right)^\frac12
\left(\mathcal{D}(r)+\frac2{N-1}\xi_f(r)\mathcal{H}(r)\right)^\frac12\\
& \leq
\xi_f(r)\frac{4}{N-1}r^{N-1}\left(\mathcal{D}(r)+\frac{N-1}2\mathcal{H}(r)\right).
\end{split}
\end{equation}
For the last term we have that 
\begin{equation}\label{bordo}
r\biggl|\int_{\partial{B_r}}fu^2\,ds\biggr|\leq \frac{\xi_f(r)}{r}\int_{\partial B_r}u^2\,ds=\xi_f(r)r^{N-1}\mathcal{H}(r).
\end{equation}
Combining \eqref{primoterm}, \eqref{fuxcdotnablau}, and \eqref{bordo},
we obtain that, for all $r\in(0,r_0)$,
\begin{equation*}
\begin{split}
|\nu_2(r)|
\leq C_1\,\xi_f(r)r^{-1}\biggl[\mathcal{N}(r)+\frac{N-1}{2}\biggr]
\end{split}
\end{equation*}
for some positive constant $C_1>0$ which does not depend on $r$. 

Now let us suppose to be under assumptions
\eqref{eta0}-\eqref{eta}. In this case, the definition of
$\mathcal R(r,u)$ allows us to rewrite $\nu_2$ as 
\begin{equation*}
\nu_2(r)=-\frac{\int_{B_r}(2f+\nabla f\cdot x)u^2\,dx}{\int_{\partial B_r}u^2\,ds}.
\end{equation*}
From \eqref{eta}, \eqref{nablaeta} and \eqref{<1/2} it follows that 
\begin{equation*}
\begin{split}
\biggl|\int_{B_r} (2f+\nabla f\cdot
x)u^2\,dx\biggr|&\leq(2\eta(r,f)+\eta(r,\nabla f\cdot x))
\biggl(\int_{B_r}|\nabla u|^2\,dx+\frac{N-1}{2r}\int_{\partial B_r}|u|^2\,ds\biggr)\\
&\leq 2(2\eta(r,f)+\eta(r,\nabla f\cdot
x)) r^{N-1}\left(\mathcal{D}(r)+\frac{N-1}2\eta(r,f)\mathcal{H}(r)+
\frac{N-1}{4}\mathcal{H}(r)\right)\\
&\leq 2(2\eta(r,f)+\eta(r,\nabla f\cdot
x))r^{N-1}\biggl(\mathcal{D}(r)+\frac{N-1}{2}\mathcal{H}(r)\biggr).
\end{split}
\end{equation*}
Therefore, we have that 
\[
|\nu_2(r)|\leq \frac{2(2\eta(r,f)+\eta(r,\nabla f\cdot
x))}r\biggl(\mathcal{N}(r)+\frac{N-1}{2}\biggr)
\]
and estimate \eqref{estimatenu2} is proved also  under assumptions
\eqref{eta0}-\eqref{eta}, with $C_1=4$.
\end{proof}
\begin{lem}\label{LemNlimitata}
Letting $r_0$ be as in Lemma \ref{lemmacoerc} and $\mathcal{N}$ as in
\eqref{frequency}, there exists a positive constant $C_2>0$ such that
\begin{equation}\label{Nlimitata}
\mathcal{N}(r)\leq C_2
\end{equation}
for all $r\in (0,r_0)$.
\end{lem}
\begin{proof}
By Lemma \ref{LemN'}, Schwarz's inequality, and Lemma
\ref{Lemstimanu2}, we obtain 
\begin{equation}\label{mathcalN+N-1'}
\biggl(\mathcal{N}+\frac{N-1}{2}\biggr)'(r)\geq \nu_2(r)\geq -C_1 \alpha(r)\biggl[\mathcal{N}(r)+\frac{N-1}{2}\biggr]
\end{equation}
for a.e. $r\in (0,r_0)$, where $\alpha$ is defined in
  \eqref{eq:alpha}. 
Taking into account that $\mathcal{N}(r)+\frac{N-1}{2}>0$
  for all $r\in(0,r_0)$ in view of \eqref{Nlimitatafrombelow}
and $\alpha\in L^1(0,r_0)$ thanks to assumptions \eqref{xiL1},
\eqref{etaL1} and \eqref{etanablafL1}, after integration over $(r,r_0)$ it follows that
\[
\mathcal{N}(r)\leq -\frac{N-1}{2}+\biggl(\mathcal{N}(r_0)+\frac{N-1}{2}\biggr)\mathrm{exp}\biggl(C_1\int_0^{r_0} \alpha(s)ds\biggr)
\]
for any $r\in (0,r_0)$, thus proving estimate \eqref{Nlimitata}.
\end{proof}
\begin{lem}\label{Lemlimesiste}
The limit 
\[
\gamma:=\lim_{r\rightarrow 0^+}\mathcal{N}(r)
\]
exists and is finite.  Moreover $\gamma\geq0$.
\end{lem}
\begin{proof}
Since 
$\mathcal{N}'(r)\geq-C_1\alpha(r)\bigl[\mathcal{N}(r)+\frac{N-1}{2}\bigr]$
in view of \eqref{mathcalN+N-1'} and $\alpha\in L^1(0,r_0)$ 
by assumptions \eqref{xiL1}, \eqref{etaL1} and \eqref{etanablafL1},
we have that 
\begin{equation*}
\frac{d}{dr}\biggl[e^{C_1\int_0^r \alpha(s)\,ds}\biggl(\mathcal{N}(r)+\frac{N-1}{2}\biggr)\biggr]\geq 0,
\end{equation*}
therefore the limit of $r\mapsto e^{C_1\int_0^r
  \alpha(s)\,ds}\bigl(\mathcal{N}(r)+\frac{N-1}{2}\bigr)$ as
$r\rightarrow 0^+$ exists; hence the function $\mathcal{N}$ 
 has a limit as $r\rightarrow 0^+$.

From \eqref{Nlimitata} and \eqref{eq:15} it follows that $C_2\geq \gamma:=\lim_{r\rightarrow
  0^+}\mathcal{N}(r)= \liminf_{r\rightarrow 0^+}\mathcal{N}(r)\geq0$;
in particular $\gamma$  is finite.
\end{proof}
A first consequence of the above analysis on the Almgren's frequency function is the following estimate of $\mathcal{H}(r)$.
\begin{lem}
Let $\gamma$ be as in Lemma
\ref{Lemlimesiste} and $r_0$ be as in Lemma \ref{lemmacoerc}. Then there exists a constant $K_1>0$ such that
\begin{equation}\label{H<kr}
\mathcal{H}(r)\leq K_1r^{2\gamma}\quad \text{for all $r\in(0,r_0)$}.
\end{equation}
On the other hand, for any $\sigma>0$ there exists a constant $K_2(\sigma)>0$ depending on $\sigma$ such that 
\begin{equation}\label{H>kr}
\mathcal{H}(r)\geq K_2(\sigma)r^{2\gamma+\sigma}\quad \text{for all $r\in(0,r_0)$}.
\end{equation}
\end{lem}
\begin{proof}
 By \eqref{mathcalN+N-1'} and \eqref{Nlimitata} we have that 
\begin{equation}\label{eq:17}
\mathcal
  N'(r)\geq -C_1\bigg(C_2+\frac{N-1}2\bigg)\alpha(r) \quad\text{a.e. in $(0,r_0)$}.
\end{equation}
Hence,
  from the fact that $\alpha\in L^1(0,r_0)$ and   
Lemma \ref{Lemlimesiste}, it follows that  $\mathcal{N}'\in
L^1(0,r_0)$. Therefore  from  
\eqref{eq:17}  it follows that 
\begin{equation}\label{N-gamma}
\mathcal{N}(r)-\gamma=\int_0^r \mathcal{N}'(s)\,ds\geq
  -C_1\bigg(C_2+\frac{N-1}2\bigg)\int_0^r\alpha(s)\,ds=-C_3rF(r),
\end{equation}
where $C_3=C_1\big(C_2+\frac{N-1}2\big)$ and 
\begin{equation*}
F(r):=\frac{1}{r}\int_0^r \alpha(s)\,ds.
\end{equation*}
 We observe that, thanks to assumptions \eqref{xiL1},
\eqref{etaL1} and \eqref{etanablafL1},
\begin{equation}\label{eq:18}
F\in L^1(0,r_0).
\end{equation}
From \eqref{H'2} and \eqref{N-gamma} we deduce that, for
a.e. $r\in (0,r_0)$,
\begin{equation*}
\frac{\mathcal{H}'(r)}{\mathcal{H}(r)}=\frac{2\mathcal{N}(r)}{r}\geq \frac{2\gamma}{r}-2C_3F(r),
\end{equation*}
which, thanks to  \eqref{eq:18},  after integration over the interval $(r,r_0)$, yields \eqref{H<kr}.

Let us prove \eqref{H>kr}. Since $\gamma:=\lim_{r\rightarrow0^+}\mathcal{N}(r)$, for any $\sigma>0$ there exists $r_\sigma>0$ such that $\mathcal{N}(r)<\gamma + \sigma/2$ for any $r\in(0,r_\sigma)$ and hence 
\begin{equation*}
\frac{\mathcal{H}'(r)}{\mathcal{H}(r)}=\frac{2\mathcal{N}(r)}{r}<\frac{2\gamma+\sigma}{r}\quad \text{for all $r\in(0,r_\sigma)$}.
\end{equation*}
Integrating over the interval $(r,r_\sigma)$ and by continuity of $\mathcal{H}$ outside $0$, we obtain \eqref{H>kr} for some constant $K_2(\sigma)$ depending on $\sigma$.
\end{proof}

\section{The Blow-up Argument}\label{sec:blow-up-argument}
In this section we develop a blow-up analysis for scaled
solutions, with the aim of classifying their possible vanishing
orders. The presence of the crack produces several additional
difficulties with respect to the classical case, mainly relying in
the persistence of the singularity even far from the origin, all along
the edge.  These  difficulties are here overcome by means of estimates
of boundary gradient integrals  (Lemma
\ref{gradwlambdaRlambdabounded}) derived by some fine doubling
properties, in the spirit of \cite{Lemmini}, where an analogous lack
of regularity far from the origin was instead produced by
many-particle and cylindrical potentials.

Throughout this section we let $u$ be a non trivial  weak
$H^1(B_{\hat{R}})$-solution to equation \eqref{eq:CPu} with $f$
satisfying either \eqref{xi0}-\eqref{xi} or
\eqref{eta0}-\eqref{eta}. Let $\mathcal{D}$ and $\mathcal{H}$ be the
functions defined in \eqref{D} and \eqref{H} and $r_0$ be as in Lemma
\ref{lemmacoerc}.
For $\lambda\in (0,r_0)$, we define the scaled function
\begin{equation}\label{ulambda}
w^\lambda(x)=\frac{u(\lambda x)}{\sqrt{\mathcal{H}(\lambda)}}.
\end{equation}
 We observe that $w^\lambda \in
H^1_{\Gamma_\lambda}(B_{\lambda^{-1}\hat{R}})$, where
\[
\Gamma_\lambda:=\lambda^{-1}\Gamma=\{x\in\R^N:\lambda
x\in\Gamma\}=\left\{
x=(x',x_N)\in\R^N:x_N\geq\frac{g(\lambda x')}{\lambda}\right\},
\] 
and 
\begin{equation*}
\int_{B_{\lambda^{-1}\hat{R}}}\nabla w^\lambda(x)\cdot\nabla v(x)\,dx-
\lambda^2\int_{B_{\lambda^{-1}\hat{R}}}f(\lambda
x)w^\lambda(x)v(x)\,dx=0\quad\text{for all } v\in
C_c^\infty(B_{\lambda^{-1}\hat{R}}\setminus \Gamma_\lambda),
\end{equation*}
i.e. $w^\lambda$ weakly solves 
\begin{equation}\label{eq:22}
\left\{\begin{aligned}
-\Delta w^\lambda(x)&=\lambda^2f(\lambda x)\,w^\lambda(x) &&\text{in}\ B_{\lambda^{-1}\hat{R}}\setminus\Gamma_\lambda, \\
w^\lambda&=0 && \text{on}\ \Gamma_\lambda.
\end{aligned}\right. 
\end{equation}
\begin{rem}\label{rem:mosco}
From assumptions \eqref{g0=0} and \eqref{gC2} we easily deduce that
$\R^{N+1}\setminus\Gamma_\lambda$ converges in the sense of Mosco (see
\cite{Daners2003,Mosco1969})  to the set 
$\R^{N+1}\setminus\tilde\Gamma$, where 
\begin{equation}\label{Gammatilde}
\tilde{\Gamma}=\{(x',x_N)\in\mathbb{R}^N:\, x_N\geq 0\}.
\end{equation}
In particular, for every $R>0$, the weak limit points in $H^1(B_R)$ as
$\lambda\to 0^+$ of the family of functions $\{w^\lambda\}_\lambda$
belong to  $H^1_{\tilde \Gamma}(B_R)$. 
\end{rem}

\begin{lem}\label{wlambdlalimitata}
For $\lambda\in (0,r_0)$, let $w^\lambda$ be defined in
\eqref{ulambda}. 
Then $\{w^\lambda\}_{\lambda\in (0,r_0)}$ is bounded in $H^1(B_1)$.
\end{lem}
\begin{proof}
From \eqref{Hcambiodivar} it follows that 
\begin{equation}\label{eq:20}
\int_{\partial
  B_1}|w^\lambda|^2 dS=1.
\end{equation} 
By scaling and \eqref{equazcoerc} we have that 
\begin{equation}\label{eq:19-a}
  \mathcal N(\lambda)\geq\frac{\lambda^{1-N}}{\mathcal
  H(\lambda)}\left(\frac{1}{2}\int_{B_\lambda}|\nabla u|^2\,dx
  -\omega(\lambda)\int_{\partial B_\lambda}u^2\,dS\right)=
\frac{1}{2}\int_{B_1}|\nabla w^\lambda(x)|^2 dx -\lambda\omega(\lambda).
\end{equation}
From \eqref{eq:19-a}, \eqref{Nlimitata}, and \eqref{eq:bound-omega} it
follows that 
\begin{equation}\label{gradwlambdalimitato}
\frac{1}{2}\int_{B_1}|\nabla w^\lambda(x)|^2 dx\leq C_2+\frac{N-1}{4}
\end{equation}
for every $\lambda\in (0,r_0)$.
The conclusion follows from \eqref{gradwlambdalimitato} and
\eqref{eq:20}, taking into account \eqref{Hardy}. 
\end{proof}
In the next lemma we prove a \textit{doubling} type result. 
\begin{lem}\label{doublingwlambda}
There exists $C_4>0$ such that
\begin{equation}\label{eq1}
\frac{1}{C_4}\mathcal{H}(\lambda)\leq \mathcal{H}(R\lambda)\leq C_4 \mathcal{H}(\lambda)\quad \text{for any $\lambda\in (0,r_0/2)$ and $R\in [1,2]$},
\end{equation}
\begin{equation}\label{eq2}
\int_{B_R}|\nabla w^\lambda(x)|^2\,dx\leq 2^{N-1}C_4\int_{B_1}|\nabla w^{R\lambda}(x)|^2\,dx\quad \text{for any $\lambda\in (0,r_0/2)$ and $R\in [1,2]$},
\end{equation}
and
\begin{equation}\label{eq3}
\int_{B_R}|w^\lambda(x)|^2\,dx\leq 2^{N+1}C_4\int_{B_1}|w^{R\lambda}(x)|^2\,dx\quad \text{for any $\lambda\in (0,r_0/2)$ and $R\in [1,2]$},
\end{equation}
where $w^\lambda$ is defined in \eqref{ulambda}.
\end{lem}
\begin{proof}
By \eqref{Nlimitatafrombelow}, \eqref{Nlimitata}, and \eqref{H'2}, it follows that
\begin{equation*}
-\frac{N-1}{2r}\leq \frac{\mathcal{H}'(r)}{\mathcal{H}(r)}=\frac{2\mathcal{N}(r)}{r}\leq \frac{2C_2}{r}\quad \text{for any $r\in (0,r_0)$}.
\end{equation*}
Let $R\in (1,2]$. For any $\lambda<r_0/R$, integrating over
$(\lambda,R\lambda)$ the above inequality and recalling that $R\leq
2$, we obtain
\[
2^{(1-N)/2}\mathcal{H}(\lambda)\leq \mathcal{H}(R\lambda)\leq 4^{C_2}\mathcal{H}(\lambda)\quad \text{for any $\lambda\in (0,r_0/R)$}.
\]
The above estimates  trivially hold also for $R=1$,
hence \eqref{eq1} with $C_4=\max\{4^{C_2},
2^{(N-1)/2}\}$  is established. 

For every $\lambda\in (0, r_0/2)$ and $R\in [1,2]$, \eqref{eq1} yields
\begin{equation*}
\begin{split}
\int_{B_R}|\nabla w^\lambda(x)|^2\,dx &= \frac{\lambda^{1-N}}{\mathcal{H}(\lambda)}\int_{B_{R\lambda}}|\nabla u(x)|^2\,dx\\
&=R^{N-1}\frac{\mathcal{H}(R\lambda)}{\mathcal{H}(\lambda)}\int_{B_1}|\nabla w^{R\lambda}(x)|^2\,dx\leq R^{N-1}C_4 \int_{B_1}|\nabla w^{R\lambda}(x)|^2\,dx,
\end{split}
\end{equation*}
thus proving \eqref{eq2}. A similar argument allows deducing
\eqref{eq3}  from \eqref{eq1}.
\end{proof}
\begin{lem}\label{lemma2}
For every $\lambda\in (0,r_0)$, let $w^\lambda$ be as in \eqref{ulambda}. Then there exist $M>0$ and $\lambda_0>0$ such that, for any $\lambda\in (0,\lambda_0)$, there exists $R_\lambda\in [1,2]$ such that 
\[
\int_{\partial B_{R_\lambda}}|\nabla w^\lambda|^2\,dS\leq M\int_{B_{R_\lambda}}|\nabla w^\lambda(x)|^2\,dx.
\]
\end{lem}
\begin{proof}
  From Lemma \ref{wlambdlalimitata} we know that the family 
  $\{w^\lambda\}_{\lambda\in (0,r_0)}$ is bounded in
  $H^1(B_1)$. Moreover Lemma \ref{doublingwlambda} implies that
  the set $\{w^\lambda\}_{\lambda\in (0,r_0/2)}$ is bounded in
  $H^1(B_2)$ and hence
\begin{equation}\label{limsupgrad}
\limsup_{\lambda\rightarrow 0^+} \int_{B_2} |\nabla w^\lambda(x)|^2 dx<+\infty.
\end{equation}
For every $\lambda\in (0,r_0/2)$, the function
$f_\lambda(r)=\int_{B_r}|\nabla w^\lambda(x)|^2 dx$ is absolutely continuous in $[0,2]$ and its distributional derivative is given by 
\[
f'_\lambda(r)=\int_{\partial B_r}|\nabla w^\lambda|^2 dS\quad \text{for a.e. $r\in (0,2)$}.
\]
We argue by contradiction and assume that for any $M>0$ there exists a sequence $\lambda_n\rightarrow 0^+$ such that 
\begin{equation*}
\int_{\partial B_r} |\nabla w^{\lambda_n}|^2dS>M\int_{B_r} |\nabla
w^{\lambda_n}(x)|^2 dx\quad\text{for all $r\in [1,2]$ and $n\in{\mathbb
  N}$},
\end{equation*}
i.e.
\begin{equation}\label{f'lambdan>Mflambdan}
f'_{\lambda_n}(r)>Mf_{\lambda_n} (r)\quad\text{for a.e. $r\in [1,2]$ and for every $n\in\mathbb{N}$}.
\end{equation}
Integration of \eqref{f'lambdan>Mflambdan} over $[1,2]$ yields
$f_{\lambda_n}(2)>e^M f_{\lambda_n}(1)$ for every $n\in\mathbb{N}$
and consequently
\[
\limsup_{n\rightarrow +\infty}f_{\lambda_n}(1)\leq e^{-M} \cdot \limsup_{n\rightarrow +\infty}f_{\lambda_n}(2).
\]
It follows that
\[
\liminf_{\lambda\rightarrow 0^+}f_\lambda(1)\leq e^{-M}\cdot \limsup_{\lambda\rightarrow 0^+}f_\lambda(2)\quad \text{for all $M>0$}.
\]
Therefore, letting $M\rightarrow +\infty$ and taking into account \eqref{limsupgrad}, we obtain that 
$\liminf_{\lambda\rightarrow 0^+}f_\lambda(1)=0$ i.e. 
\begin{equation}\label{eq:19}
\liminf_{\lambda\rightarrow 0^+} \int_{B_1} |\nabla w^\lambda(x)|^2 dx=0.
\end{equation}
From \eqref{eq:19} and boundedness of $\{w^\lambda\}_{\lambda\in (0,r_0)}$ in
  $H^1(B_1)$ there exist a sequence $\tilde{\lambda}_n\rightarrow 0$
  and  some $w\in
H^1(B_1)$
  such that
$w^{\tilde{\lambda}_n}\rightharpoonup w$ in $H^1(B_1)$  and 
\begin{equation}\label{gradlambdatilde0}
\lim_{n\rightarrow +\infty}\int_{B_1}|\nabla w^{\tilde{\lambda}_n}(x)|^2 dx=0.
\end{equation}
The compactness of the trace map from $H^1(B_1)$
to $L^2(\partial B_1)$ and \eqref{eq:20} imply that 
\begin{equation}\label{eq:21}
\int_{\partial B_1}|w|^2 dS=1.
\end{equation}
 Moreover, by weak lower semicontinuity and \eqref{gradlambdatilde0}, 
\begin{equation*}
\int_{B_1} |\nabla w(x)|^2 dx\leq \lim_{n\rightarrow +\infty}\int_{B_1} |\nabla w^{\tilde{\lambda}_n}(x)|^2 dx=0.
\end{equation*}
 Hence $w\equiv \text{\rm const}$ in $B_1$. On the other hand, in view
of Remark \ref{rem:mosco},   $w\in H^1_{\tilde\Gamma}(B_1)$ so that
$w\equiv 0$ in $B_1$,
thus contradicting \eqref{eq:21}.
\end{proof}
\begin{lem}\label{gradwlambdaRlambdabounded}
Let $w^\lambda$ be as in \eqref{ulambda} and $R_\lambda$ be as in  Lemma \ref{lemma2}. Then there exists $\overline{M}$ such that 
\begin{equation*}
\int_{\partial B_1} |\nabla w^{\lambda R_\lambda}|^2\,dS\leq \overline{M} \quad\text{for any $0<\lambda<\min\Bigl\{\lambda_0,\frac{r_0}{2}\Bigr\}$}.
\end{equation*}
\end{lem}
\begin{proof}
Since 
\begin{equation*}
  \int_{\partial B_1} |\nabla w^{\lambda R_\lambda}|^2\,dS=\frac{\lambda^2 R_\lambda^{2-N}}{\mathcal{H}(\lambda R_\lambda)}\int_{\partial B_{R_\lambda}}|\nabla u(\lambda x)|^2\,dS=\frac{R_\lambda^{2-N}\mathcal{H}(\lambda)}{\mathcal{H}(\lambda R_\lambda)}\int_{\partial B_{R_\lambda}}|\nabla w^\lambda|^2\,dS,
\end{equation*}
from  \eqref{eq1}, \eqref{eq2}, Lemma \ref{lemma2}, Lemma
\ref{wlambdlalimitata}, and the fact that $1\leq R_\lambda\leq 2$, we
deduce that, for every $0<\lambda<\min\{\lambda_0,\frac{r_0}{2}\}$,
\begin{equation*}
\int_{\partial B_1} |\nabla w^{\lambda R_\lambda}|^2\,dS \leq C_4 M \int_{B_{R_\lambda}}|\nabla w^{\lambda}(x)|^2\,dx\leq 2^{N-1}C_4^2 M\int_{B_1}|\nabla w^{\lambda R_\lambda}(x)|^2\,dx\leq \overline{M}<+\infty,
\end{equation*}
thus completing the proof. 
\end{proof}
\begin{lem}\label{Lemmablow-up}
Let 
$u\in H^1(B_{\hat{R}})\setminus\{0\}$ be a non-trivial weak solution to
  \eqref{eq:CPu} with $f$ satisfying  either \eqref{xi0}\eqref{xi} or \eqref{eta0}\eqref{eta}. Let $\gamma$ be as in Lemma \ref{Lemlimesiste}. Then
\begin{itemize}
\item[$(i)$] there exists $k_0\in\mathbb{N}\setminus\{0\}$
 such that $\gamma=\frac{k_0}2$; 
\item[$(ii)$] for every sequence $\lambda_n \rightarrow 0^+$, there
  exist a subsequence 
$\{\lambda_{n_k}\}_{k\in\mathbb{N}}$ and an eigenfunction $\psi$ of
 problem \eqref{eq:1} 
associated with the eigenvalue $\mu_{k_0}$
such that $\Vert \psi\Vert_{L^2(\mathbb{S}^N)}=1$ and 
\begin{equation}\label{eq:29}
\frac{u(\lambda_{n_k}x)}{\sqrt{\mathcal{H}(\lambda_{n_k})}}\rightarrow
|x|^\gamma\psi\biggl(\frac{x}{|x|}\biggr)
\quad\text{strongly in $H^1(B_1)$}.
\end{equation}
\end{itemize} 
\end{lem}
\begin{proof}
For $\lambda\in (0,\min\{r_0,\lambda_0\})$, let $w^\lambda$ be as in
\eqref{ulambda} and $R_\lambda$
  be as in Lemma \ref{lemma2}. Let $\lambda_n \rightarrow 0^+$. By Lemma \ref{wlambdlalimitata}, we have that the set
  $\{w^{\lambda R_\lambda}: \lambda\in (0,\min\{r_0/2,\lambda_0\})\}$ is bounded in
  $H^1(B_1)$. Then there exists a  subsequence
  $\{\lambda_{n_k}\}_k$ 
  such that
  $w^{\lambda_{n_k}R_{\lambda_{n_k}}}\rightharpoonup w$ weakly in $H^1(B_1)$
  for some function $w\in H^1(B_1)$. The compactness of the trace
  map from $H^1(B_1)$ into $L^2(\partial B_1)$ and \eqref{eq:20} ensure that
\begin{equation}\label{eq:25}  
\int_{\partial B_1}|w|^2 dS=1
\end{equation}
and, consequently, $w\not\equiv
  0$.  Furthermore, in view of Remark \ref{rem:mosco} we have that
  $w\in H^1_{\tilde \Gamma}(B_1)$, where  $\tilde{\Gamma}$ is the set
  defined in \eqref{Gammatilde}. 

 Let  $\phi\in C^\infty_c(B_1\setminus \tilde{\Gamma})$. 
It is easy to verify that $\phi\in C^\infty_c(B_1\setminus \Gamma_\lambda)$
provided $\lambda$ is sufficiently small. Therefore, since
$w^{\lambda_{n_k}R_{\lambda_{n_k}}}$ weakly satisfies equation 
\eqref{eq:22} with $\lambda=\lambda_{n_k}R_{\lambda_{n_k}}$ and, for
sufficiently large $k$, $B_1\subset
B_{(\lambda_{n_k}R_{\lambda_{n_k}})^{-1}\hat R}$, we have that 
\begin{equation}\label{wlambdankRlambdank}
\int_{B_1}\nabla w^{\lambda_{n_k}R_{\lambda_{n_k}}}\cdot\nabla \phi\,dx-
(\lambda_{n_k}R_{\lambda_{n_k}})^2\int_{B_1}f(\lambda_{n_k}R_{\lambda_{n_k}}
x)w^{\lambda_{n_k}R_{\lambda_{n_k}}} \phi\,dx=0
\end{equation}
for $k$ sufficiently large. 

Under the set of assumptions \eqref{xi0}-\eqref{xi}, from
\eqref{Hardy} it follows that
\begin{multline}\label{lambdaftendea01}
\lambda^2\biggl|\int_{B_1}f(\lambda x)w^\lambda(x)\phi(x)\,dx\biggr|
\leq \xi_f(\lambda) \biggl(\int_{B_1}\frac{|w^\lambda(x)|^2}{|x|^2}dx\biggr)^{\!\!1/2}\biggl(\int_{B_1}\frac{|\phi(x)|^2}{|x|^2}dx\biggr)^{\!\!1/2}\\
\leq \frac{4 \xi_f(\lambda)}{(N-1)^2} \biggl(\int_{B_1}|\nabla
w^\lambda|^2 dx+\frac{N-1}{2} \biggr)^{\!\!1/2}\biggl(\int_{B_1}
|\nabla \phi|^2 dx
 \biggr)^{\!\!1/2}=o(1) 
\end{multline}
as $\lambda\rightarrow 0^+$.
Similarly, under assumptions \eqref{eta0}-\eqref{eta}, by scaling, we
obtain that, as $\lambda\rightarrow 0^+$,
\begin{multline}\label{lambdaftendea02}
\lambda^2\biggl|\int_{B_1}f(\lambda x)w^\lambda(x)\phi(x)\,dx\biggr|
\\\leq \eta(\lambda,f) \biggl(\int_{B_1}|\nabla w^\lambda|^2
dx+\frac{N-1}{2}\biggr)^{\!\!1/2}
\biggl(
\int_{B_1}
|\nabla \phi|^2 dx\biggr)^{\!\!1/2}
=o(1).
\end{multline}
The weak convergence  of  $w^{\lambda_{n_k}R_{\lambda_{n_k}}}$ to $w$ in $H^1(B_1)$
  and \eqref{lambdaftendea01}-\eqref{lambdaftendea02} allow passing to the limit in \eqref{wlambdankRlambdank}
thus yielding that $w\in H^1_{\tilde \Gamma}(B_1)$  satisfies
\begin{equation*}
\int_{B_1}\nabla w(x)\cdot\nabla \phi(x)\,dx=0\quad\text{for all } \phi\in
C_c^\infty(B_1\setminus \tilde{\Gamma}),
\end{equation*}
i.e. $w$ weakly solves 
\begin{equation}\label{warmonica}
\left\{\begin{aligned}
    -\Delta w(x)&=0 &&\text{in}\ B_1\setminus \tilde{\Gamma}, \\
w&=0 && \text{on}\ \tilde{\Gamma}.
\end{aligned}\right. 
\end{equation}
We observe that, by classical regularity theory, $w$ is smooth in
$B_1\setminus \tilde{\Gamma}$.

From Lemma \ref{gradwlambdaRlambdabounded} and the density 
of $C^\infty(\overline{B_1}\setminus\tilde{\Gamma})$ in $H^1_{\tilde{\Gamma}}(B_1)$, it follows that
\begin{equation}\label{wlambdankphi}
\int_{B_1}\nabla w^{\lambda_{n_k}R_{\lambda_{n_k}}}\!\cdot\!\nabla \phi\,dx=\lambda_{n_k}^2 R_{\lambda_{n_k}}^2\int_{B_1}f(\lambda_{n_k}R_{\lambda_{n_k}}x)w^{\lambda_{n_k}R_{\lambda_{n_k}}}\phi\,dx+\int_{\partial B_1}\frac{\partial w^{\lambda_{n_k}R_{\lambda_{n_k}}}}{\partial\nu}\phi\,dS
\end{equation}
 for every $\phi\in H^1_{\tilde{\Gamma}}(B_1)$ as well as for every $\phi\in H^1_{\Gamma_{\lambda_{n_k}R_{\lambda_{n_k}}}}(B_1).$
From Lemma \ref{gradwlambdaRlambdabounded} it follows that, up to
a subsequence still denoted as $\{\lambda_{n_k}\}$, 
there exists $g\in L^2(\partial B_1)$ such that 
\begin{equation}\label{eq:24}
\frac{\partial
  w^{\lambda_{n_k}R_{\lambda_{n_k}}}}{\partial\nu}\rightharpoonup
g\quad \text{weakly in $L^2(\partial B_1)$}.
\end{equation}
Passing to the limit in \eqref{wlambdankphi} and taking into account
\eqref{lambdaftendea01}-\eqref{lambdaftendea02}, we then obtain that 
\begin{equation*}
\int_{B_1}\nabla w\cdot\nabla \phi\,dx=\int_{\partial
  B_1}g\,\phi\,dS\quad\text{for every } \phi\in H^1_{\tilde{\Gamma}}(B_1).
\end{equation*}
In particular, taking $\phi=w$ above, we have that 
\begin{equation}\label{eq:23}
  \int_{B_1}|\nabla w|^2\,dx=\int_{\partial
  B_1}g\,w\,dS.
\end{equation}
On the other hand, 
from \eqref{wlambdankphi} with $\phi=
w^{\lambda_{n_k}R_{\lambda_{n_k}}}$,
\eqref{lambdaftendea01}-\eqref{lambdaftendea02}, \eqref{eq:24},
the weak convergence  of  $w^{\lambda_{n_k}R_{\lambda_{n_k}}}$ to $w$ in $H^1(B_1)$
(which implies the strong convergence  of  the traces in $L^2(\partial
B_1)$ by compactness of the trace
  map from $H^1(B_1)$ into $L^2(\partial B_1)$), and \eqref{eq:23} it follows that 
\begin{align*}
\lim_{k\to+\infty}
&\int_{B_1}|\nabla w^{\lambda_{n_k}R_{\lambda_{n_k}}}|^2\,dx\\
&=\lim_{k\to+\infty}
\bigg(\lambda_{n_k}^2
  R_{\lambda_{n_k}}^2\int_{B_1}f(\lambda_{n_k}R_{\lambda_{n_k}}x)|w^{\lambda_{n_k}R_{\lambda_{n_k}}}|^2\,dx+\int_{\partial
  B_1}\frac{\partial
  w^{\lambda_{n_k}R_{\lambda_{n_k}}}}{\partial\nu}w^{\lambda_{n_k}R_{\lambda_{n_k}}}\,dS\bigg)\\
&=
\int_{\partial B_1}gw\,dS=  \int_{B_1}|\nabla w|^2\,dx
\end{align*}
which implies that 
\begin{equation}\label{convergenzaforte}
w^{\lambda_{n_k}R_{\lambda_{n_k}}}\rightarrow w \quad\text{strongly in $H^1(B_1)$}.
\end{equation}
For every $k\in\mathbb{N}$ and $r\in (0,1]$, let 
\[
\mathcal{D}_k(r)= r^{1-N}\int_{B_r} \Bigl(|\nabla w^{\lambda_{n_k}R_{\lambda_{n_k}}}(x)|^2 -\lambda_{n_k}^2 R_{\lambda_{n_k}}^2f(\lambda_{n_k}R_{\lambda_{n_k}}x)|w^{\lambda_{n_k}R_{\lambda_{n_k}}}(x)|^2\Bigr)\,dx\\
\]
and 
\[
\mathcal{H}_k(r)=r^{-N}\int_{\partial B_r}|w^{\lambda_{n_k}R_{\lambda_{n_k}}}|^2\,dS.
\]
We also define, for all $r\in(0,1]$, 
\begin{equation*}
\mathcal{D}_w(r)=r^{1-N}\int_{B_r}|\nabla w|^2\,dx\quad\text{and}\quad
\quad \mathcal{H}_w(r)=r^{-N} \int_{\partial B_r} |w|^2\,dS.
\end{equation*}
A change of variables directly gives 
\begin{equation}\label{Nk}
\mathcal{N}_k(r):=\frac{\mathcal{D}_k(r)}{\mathcal{H}_k(r)}=\frac{\mathcal{D}(\lambda_{n_k}R_{\lambda_{n_k}}r)}{\mathcal{H}(\lambda_{n_k}R_{\lambda_{n_k}}r)}=\mathcal{N}(\lambda_{n_k}R_{\lambda_{n_k}}r)\quad \text{for all $r\in (0,1]$}. 
\end{equation}
From \eqref{convergenzaforte}, \eqref{lambdaftendea01}-\eqref{lambdaftendea02} and  compactness of the trace map from
$H^1(B_r)$ into $L^2(\partial B_r)$, it follows that, for every fixed $r\in (0,1]$,
\begin{equation}\label{Dk-Hk}
\mathcal{D}_k(r)\rightarrow \mathcal{D}_w(r)\quad\text{and}\quad 
\mathcal{H}_k(r)\rightarrow \mathcal{H}_w(r).
\end{equation}
We
observe that $\mathcal H_w(r)>0$ for all $r\in(0,1]$; indeed if, for
some $r\in (0,1]$, $\mathcal{H}_w(r)=0$, then $w=0$ on $\partial B_r$
and, testing \eqref{warmonica} with
$w\in H^1_0(B_r\setminus\tilde{\Gamma})$, we would obtain
$\int_{B_r}|\nabla w|^2\,dx=0$ and hence $w\equiv0$ in $B_r$, thus
contradicting classical unique continuation principles for second order elliptic
equations (see e.g. \cite{Wolff}). 
Therefore the function 
\[
\mathcal{N}_w:(0,1]\to\R,\quad \mathcal{N}_w(r):=\frac{\mathcal{D}_w(r)}{\mathcal{H}_w(r)}
\]
is well defined. Moreover \eqref{Nk}, \eqref{Dk-Hk}, and Lemma
\ref{Lemlimesiste}, imply that, for all $r\in (0,1]$, 
\begin{equation}\label{Nw}
\mathcal{N}_w(r)=\lim_{k\rightarrow\infty} \mathcal{N}(\lambda_{n_k}R_{\lambda_{n_k}}r)=\gamma.
\end{equation}
 Therefore $\mathcal{N}_w$ is constant in $(0,1]$ and hence
 $\mathcal{N}'_w(r)=0$ for any $r\in (0,1)$. Hence, from
 \eqref{warmonica} and Lemma \ref{LemN'} with $f\equiv 0$,  we deduce
 that, for a.e. $r\in(0,1)$, 
\begin{equation*}
0=\mathcal{N}_w'(r)\geq
\nu_1(r)=\frac{2r\bigl[\bigl(\int_{\partial B_r}\bigl|\frac{\partial w}{\partial \nu}\bigr|^2 \,dS\bigr)\bigl(\int_{\partial B_r}|w|^2\,dS\bigr)-\bigl(\int_{\partial B_r}w\frac{\partial w}{\partial\nu}\,dS\bigr)^2\bigr]}{\bigl(\int_{\partial B_r}|w|^2\,dS\bigr)^2}\geq0
\end{equation*}
so that $\bigl(\int_{\partial B_r}\bigl|\frac{\partial w}{\partial \nu}\bigr|^2 \,dS\bigr)\bigl(\int_{\partial B_r}|w|^2\,dS\bigr)-\bigl(\int_{\partial B_r}w\frac{\partial w}{\partial\nu}\,dS\bigr)^2=0$.
 This implies that $w$ and $\frac{\partial w}{\partial\nu}$ have the
same direction as vectors in $L^2(\partial B_r)$ for a.e. $r\in
(0,1)$.   Then there exists 
a function $\zeta=\zeta(r)$, defined   a.e. in $(0,1)$, such that 
$\frac{\partial w}{\partial\nu}(r\theta)=\zeta(r)w(
  r\theta)$ for a.e. $r\in (0,1)$ and for all
  $\theta\in\mathbb{S}^N\setminus\Sigma$.
Multiplying by $w(r\theta)$ and integrating over $\mathbb{S}^N$ we
obtain that 
\[
\int_{\mathbb{S}^N}\frac{\partial w}{\partial\nu}(r\theta)\,w
(r\theta)\,dS=\zeta(r)
\int_{\mathbb{S}^N} w^2(r\theta)\,dS
\]
and hence, in view of \eqref{H'} and \eqref{Hcambiodivar}, 
$\zeta(r)=\frac{\mathcal{H}'_w(r)}{2\mathcal{H}_w(r)}$ for a.e $r\in
(0,1)$. This in particular implies that $\zeta\in
L^1_{\mathrm{loc}}(0,1]$. Moreover, after integration, we obtain 
\begin{equation*}
w(r\theta)=e^{\int_1^r \zeta(s)ds}w(1\theta)=\varphi(r)\psi(\theta)\quad \text{for all $r\in (0,1),\,\theta\in\mathbb{S}^N\setminus\Sigma$}, 
\end{equation*}
where $\varphi(r)=e^{\int_1^r \zeta(s)ds}$ and
$\psi=w\big|_{\sfera^N}$.
The fact that $w\in H^1_{\tilde \Gamma}(B_1)$
 implies that $\psi\in H^1_0(\sfera^N\setminus\Sigma)$; moreover
\eqref{eq:25} yields that 
\begin{equation}\label{eq:26}
\int_{\sfera^N}\psi^2(\theta)\,dS=1.
\end{equation}
 Equation \eqref{warmonica} rewritten in polar coordinates
$r,\theta$ becomes
\[
\biggl(-\varphi''(r)-\frac{N}{r}\varphi'(r)\biggr)\psi(\theta)
-\frac{\varphi(r)}{r^2}\Delta_{\mathbb{S}^N}\psi(\theta)=0 \quad\text{on
}\sfera^N\setminus\Sigma.
\]
The above equation for a fixed $r$  implies that $\psi$ is an
eigenfunction of problem \eqref{eq:1}. Letting $\mu_{k_0}=\frac{k_0(k_0+2N-2)}4$ be the
corresponding eigenvalue, $\varphi$ solves
\[
-\varphi''(r)-\frac{N}{r}\varphi'(r)+\frac{\mu_{k_0}}{r^2} \varphi(r)=0.
\]
Integrating the above equation we obtain that there exist
$c_1,c_2\in\R$ such that 
\[
\varphi(r)=c_1r^{\sigma_{k_0}^+}+c_2 r^{\sigma_{k_0}^-},
\]
where 
\[
\sigma_{k_0}^+=-\frac{N-1}{2}+\sqrt{\biggl(\frac{N-1}{2}\biggr)^2+\mu_{k_0}}=\frac{k_0}2
\]
and 
\[
\sigma_{k_0}^-=-\frac{N-1}{2}-\sqrt{\biggl(\frac{N-1}{2}\biggr)^2+\mu_{k_0}}=
-\big(N-1+\tfrac{k_0}2\big).
\]
Since the function
$|x|^{\sigma_{k_0}^-}\psi\bigl(\frac{x}{|x|}\bigr)\notin
L^{2^\ast}(B_1)$ (where $2^\ast=2(N+1)/(N-1)$), we have that
$|x|^{\sigma_{k_0}^-}\psi\bigl(\frac{x}{|x|}\bigr)$ does not belong
to $H^1(B_1)$; then necessarily $c_2=0$ and $\varphi(r)=c_1r^{k_0/2}$. Since $\varphi(1)=1$, we obtain that $c_1=1$ and then
\begin{equation}\label{c2=0}
w(r\theta)=r^{k_0/2}\psi(\theta),\quad\text{for all $r\in (0,1)$ and $\theta\in\mathbb{S}^N\setminus\Sigma$}.
\end{equation}
Let us now consider the sequence $\{w^{\lambda_{n_k}}\}$. Up to a
further subsequence still denoted by $w^{\lambda_{n_k}}$, we may
suppose that $w^{\lambda_{n_k}}\rightharpoonup \overline{w}$ weakly in
$H^1(B_1)$ for some $\overline{w}\in H^1(B_1)$ and that
$R_{\lambda_{n_k}}\rightarrow \overline{R}$ for some
$\overline{R}\in [1,2]$. Strong convergence of
$w^{\lambda_{n_k}R_{\lambda_{n_k}}}$ in $H^1(B_1)$ implies that, up to
a subsequence, both $w^{\lambda_{n_k}R_{\lambda_{n_k}}}$ and
$|\nabla w^{\lambda_{n_k}R_{\lambda_{n_k}}}|$ are dominated by a
$L^2(B_1)$-function uniformly with respect to $k$. Furthermore, in view of \eqref{eq1}, up to a subsequence we can assume that the limit
\[
\ell:=\lim_{k\rightarrow +\infty} \frac{\mathcal{H}(\lambda_{n_k}R_{\lambda_{n_k}})}{\mathcal{H}(\lambda_{n_k})}
\]
exists and is finite. The Dominated Convergence Theorem then implies
\begin{equation*}
\begin{split}
&\lim_{k\rightarrow +\infty}\int_{B_1} w^{\lambda_{n_k}}(x)v(x)\,dx=\lim_{k\rightarrow +\infty}R_{\lambda_{n_k}}^{N+1} \int_{B_{1/R_{\lambda_{n_k}}}} w^{\lambda_{n_k}}(R_{\lambda_{n_k}}x)v(R_{\lambda_{n_k}}x)\,dx\\
&=\lim_{k\rightarrow +\infty}R_{\lambda_{n_k}}^{N+1} \sqrt{\frac{\mathcal{H}(\lambda_{n_k}R_{\lambda_{n_k}})}{\mathcal{H}(\lambda_{n_k})}}\int_{B_1} \chi_{B_{1/R_{\lambda_{n_k}}}}(x)w^{\lambda_{n_k}R_{\lambda_{n_k}}}(x)v(R_{\lambda_{n_k}}x)\,dx\\
&=\overline{R}^{N+1}\sqrt{\ell}\int_{B_1} \chi_{B_{1/\overline{R}}}(x)w(x)v(\overline{R}x)\,dx=\overline{R}^{N+1}\sqrt{\ell}\int_{B_{1/\overline{R}}}w(x)v(\overline{R}x)\,dx=\sqrt{\ell}\int_{B_1} w(x/\overline{R})v(x)\,dx
\end{split}
\end{equation*}
for any $v\in C^\infty_{\rm c}(B_1)$. By density it is easy to verify
that the previous convergence also holds for all $v\in L^2(B_1)$. We
conclude that $w^{\lambda_{n_k}}\rightharpoonup
\sqrt{\ell}\,w(\cdot/{\overline{R}})$ weakly in $L^2(B_1)$;
as a consequence we have that
$\overline{w}=\sqrt{\ell}\,w\big(\frac\cdot{\overline{R}}\big)$ and 
$w^{\lambda_{n_k}}\rightharpoonup
\sqrt{\ell}\,w(\cdot/{\overline{R}})$ weakly in $H^1(B_1)$.
Moreover 
\begin{equation*}
\begin{split}
&\lim_{k\rightarrow +\infty}\int_{B_1}|\nabla w^{\lambda_{n_k}}(x)|^2\,dx=\lim_{k\rightarrow +\infty}R_{\lambda_{n_k}}^{N+1} \int_{B_{1/R_{\lambda_{n_k}}}} |\nabla w^{\lambda_{n_k}}(R_{\lambda_{n_k}}x)|^2\,dx\\
&=\lim_{k\rightarrow +\infty}R_{\lambda_{n_k}}^{N-1} \frac{\mathcal{H}(\lambda_{n_k}R_{\lambda_{n_k}})}{\mathcal{H}(\lambda_{n_k})}\int_{B_1} \chi_{B_{1/R_{\lambda_{n_k}}}}(x)|\nabla w^{\lambda_{n_k}R_{\lambda_{n_k}}}(x)|^2\,dx\\
&=\overline{R}^{N-1}\ell\int_{B_1} \chi_{B_{1/\overline{R}}}(x)|\nabla w(x)|^2\,dx=\overline{R}^{N-1}\ell\int_{B_{1/\overline{R}}}|\nabla w(x)|^2\,dx=\int_{B_1}|\sqrt{\ell}\, \nabla (w(x/\overline{R}))|^2\,dx.
\end{split}
\end{equation*}
Therefore we conclude that $w^{\lambda_{n_k}}\rightarrow
\overline{w}=\sqrt{\ell}w(\cdot/\overline{R})$ strongly in
$H^1(B_1)$. Furthermore, by \eqref{c2=0} and the fact that
$\int_{\partial B_1}|\overline{w}|^2\,dS=\int_{\partial
  B_1}|w|^2\,dS=1$, we deduce that $\overline{w}=w$.

It remains to prove part (i).  From 
\eqref{c2=0} and \eqref{eq:26} it follows that
$H_w(r)=r^{k_0}$. Therefore 
\eqref{Nw} and
Lemma \ref{LemH'} applied to $w$ imply that 
\[
\gamma=\frac r2\, \frac{H_w'(r)}{H_w(r)}=\frac r2\,\frac{k_0
  \,r^{k_0-1}}{r^{k_0}}=\frac{k_0}2, 
\]
thus completing the proof.
\end{proof}
In order to make more explicit the blow-up result proved above, we are
going to describe the asymptotic behavior of $\mathcal{H}(r)$ as $r\rightarrow 0^+$.
\begin{lem}\label{LemmaesistelimH}
Let $\gamma$ be as in Lemma \ref{Lemlimesiste}. The limit
$\lim_{r\rightarrow 0^+}r^{-2\gamma}\mathcal{H}(r)$ exists and it is finite. 
\end{lem}
\begin{proof}
Thanks to estimate \eqref{H<kr}, it is enough to prove that the limit exists. By  \eqref{H'2} and Lemma \ref{Lemlimesiste} we have 
\begin{equation}\label{eq:27}
\frac{d}{dr}\frac{\mathcal{H}(r)}{r^{2\gamma}}=2r^{-2\gamma-1}(\mathcal{D}(r)-\gamma\mathcal{H}(r))=2r^{-2\gamma-1}\mathcal{H}(r)\int_0^r \mathcal{N}'(s)\,ds. 
\end{equation}
Let us write $\mathcal N'=\alpha_1+\alpha_2$,
where, using the same notation as in Section
\ref{sec:almgr-type-freq},  
\[
\alpha_1(r)=\mathcal
  N'(r)+C_1\bigg(C_2+\frac{N-1}2\bigg)\alpha(r)\quad\text{and}\quad 
\alpha_2=-C_1\bigg(C_2+\frac{N-1}2\bigg)\alpha(r).
\]
From \eqref{eq:17} we have that  $\alpha_1(r)\geq 0$ for
a.e. $r\in(0,r_0)$.  Moreover \eqref{eq:alpha} and assumptions \eqref{xiL1},
\eqref{etaL1} and \eqref{etanablafL1} ensure that $\alpha_2\in
L^1(0,r_0)$ and 
\begin{equation}\label{eq:28}
\frac1{s}\int_0^s \alpha_2(t)\,dt\in L^1(0,r_0).
\end{equation}
 Integration of \eqref{eq:27}
over $(r,r_0)$ yields
\begin{equation}\label{tuttoinL1}
\frac{\mathcal{H}(r_0)}{r_0^{2\gamma}}-\frac{\mathcal{H}(r)}{r^{2\gamma}}= \int_r^{r_0}2s^{-2\gamma-1}\mathcal{H}(s)\biggl(\int_0^s \alpha_1(t)dt\biggr)\,ds+\int_r^{r_0}2s^{-2\gamma-1}\mathcal{H}(s)\biggl(\int_0^s \alpha_2(t)dt\biggr)\,ds.
\end{equation}
Since $\alpha_1(t)\geq 0$ we have that $\lim_{r\rightarrow
  0^+}\int_r^{r_0} 2s^{-2\gamma-1}\mathcal{H}(s)\biggl(\int_0^s
\alpha_1(t)dt\biggr)\,ds$ exists. On the other hand, \eqref{H<kr} and
\eqref{eq:28} imply that
\begin{equation*}
\biggl|s^{-2\gamma-1}\mathcal{H}(s)\biggl(\int_0^s \alpha_2(t)dt\biggr)\,ds\biggr|\leq K_1s^{-1}\int_0^s \alpha_2(t)\,dt\in L^1(0,r_0)
\end{equation*}
for all $s\in (0,r_0)$, thus  proving that
$s^{-2\gamma-1}\mathcal{H}(s)\bigl(\int_0^s \alpha_2(t)dt\bigr)\in
L^1(0,r_0)$. 
Then we may conclude that both terms in the right hand side of
\eqref{tuttoinL1} admit a limit as $r\rightarrow 0^+$
and at least one of such limits is finite,
 thus completing the proof of the lemma.
\end{proof}
\section{Straightening the domain}\label{sec:straightening-domain}
In order to detect the sharp vanishing order of the function
$\mathcal H$ and to give a more explicit blow-up result,  in this section we construct an auxiliary
equivalent problem by a diffeomorphic deformation of the domain,  
 inspired by \cite{Almgren-type}, see also  \cite{Adolf-Escau} and 
  \cite{Tao}. The purpose of such deformation is to straighten
  the crack; the advantage of working in a domain with a straight crack will
  then rely in the possibility of  separating radial and angular
  coordinates in the Fourier expansion of solutions (see \eqref{vespansion}).

\begin{lem}\label{Straightening}
There exists $\bar{r}\in (0,r_0)$ such that the function
\[
\Xi\colon B_{\bar{r}}\to B_{\bar{r}},
\]
\begin{equation*}
\Xi(y)=\Xi(y',y_N,y_{N+1})=
\begin{cases}
\dfrac{(y',y_N-g(y'),y_{N+1})}{\sqrt{1+\frac{g^2
      (y')-2g(y')y_N}{|y'|^2+y_N^2+y_{N+1}^2}}},&\text{if }y\neq0,\\[15pt]
0,&\text{if }y=0,
\end{cases}
\end{equation*}
is a  $C^1$-diffeomorphism. Furthermore, setting $\Phi=\Xi^{-1}$, we have that
\begin{align}\label{PhiC1}
&\varPhi(
  B_{{r}}\setminus\tilde{\Gamma})=B_{{r}}\setminus\Gamma,
\quad \varPhi^{-1}(B_{r}\setminus \Gamma)=B_{r}\setminus\tilde{\Gamma}\quad\text{for all
  }r\in(0,\bar r),\\
&\label{Phisulbordo}
\varPhi(\partial B_r)= \partial B_r\quad\text{for all
}r\in(0,\bar r),\\
&\label{Phisviluppo}
\varPhi(x)=x+O(|x|^2)\quad\text{and}\quad \mathrm{Jac}\,\varPhi(x)=\mathrm{Id}_{N+1}+O(|x|)\quad \text{as $|x|\rightarrow 0$},\\
&\label{Phi-1sviluppo}
\varPhi^{-1}(y)=y+O(|y|^2)\quad\text{and}\quad \mathrm{Jac}\,\varPhi^{-1}(y)=\mathrm{Id}_{N+1}+O(|y|)\quad \text{as $|y|\rightarrow 0$},\\
&\label{detJac}
\mathrm{det\,Jac}\,\varPhi(x)=1+O(|x|)
\quad\text{and}\quad
\mathrm{det\,Jac}\,\varPhi^{-1}(y)=1+O(|y|)
 \quad \text{as $|x|\rightarrow 0$, $|y|\rightarrow 0$}. 
\end{align}
\end{lem}
\begin{proof}
The proof follows from the local inversion theorem,
\eqref{g0=0},\eqref{gC2}, \eqref{gO2}, and direct calculations.  
\end{proof}
 Let $u\in H^1(B_{\hat{R}})$ be a weak solution to \eqref{eq:CPu}. Then 
\begin{equation}\label{v}
v=u\circ\varPhi\in H^1(B_{\bar{r}})
\end{equation}
is a weak solution to 
\begin{equation}\label{eq:CPv}
\left\{\begin{aligned}
-\text{div}(A(x)\nabla v(x))&=\tilde{f}(x)v(x) &&\text{in }B_{\bar{r}}\setminus\tilde{\Gamma}, \\
v&=0 && \text{on}\ \tilde{\Gamma},
\end{aligned}\right. 
\end{equation}
 i.e. 
\begin{equation*}
\left\{\begin{aligned}
&v\in H^1_{\tilde{\Gamma}}(B_{\bar r}), \\
&\int_{B_{\bar r}} A(x)\nabla v(x)\cdot\nabla \varphi(x)\,dx-\int_{B_{\bar
    r}}\tilde f(x)v(x)\varphi(x)\,dx=0\quad\text{for any}\ \varphi\in
C_c^\infty(B_{\bar r}\setminus\tilde\Gamma).
\end{aligned}\right. 
\end{equation*}
where
\begin{equation}\label{eq:def-f}
A(x)=|\mathrm{det\,Jac}\,\varPhi(x)|(\text{Jac}\,\varPhi(x))^{-1}((\text{Jac}\,\varPhi(x))^T)^{-1},\quad
\tilde{f}(x)=|\mathrm{det\,Jac}\,\varPhi(x)|f(\varPhi(x)).
\end{equation}
By Lemma \ref{Straightening} and direct calculations, we obtain that
\begin{equation}\label{tildeA}
A(x)=\text{Id}_{N+1}+ O(|x|)\quad \text{as $|x|\rightarrow 0$}.
\end{equation}
\begin{lem}\label{Hdellav}
Letting $\mathcal{H}$ be as in \eqref{H} and $v=u\circ \varPhi$ as in
\eqref{v}, we have that 
\begin{equation}\label{Hlambda}
\mathcal{H}(\lambda)=(1+O(\lambda))\int_{\mathbb{S}^N}v^2(\lambda\theta)\,dS\quad\text{as
}\lambda\to0^+,
\end{equation}
and 
\begin{equation}\label{gradvlambda}
\frac{\int_{B_1}|\nabla \hat{v}^\lambda(x)|^2 dx}{\mathcal{H}(\lambda)}=(1+O(\lambda))\int_{B_1}|\nabla w^\lambda(y)|^2 dy=O(1) \quad\text{as
}\lambda\to0^+,
\end{equation}
where $w^\lambda$ is defined in \eqref{ulambda} and $\hat{v}^\lambda(x):=v(\lambda x)$. 
\end{lem}
\begin{proof}
From  \eqref{PhiC1} and a change of variable  it follows that 
\[
\int_{B_\lambda}u^2(x)\,dx=
\int_{B_\lambda}v^2(y) |\mathrm{det\,Jac}\,\varPhi (y)|\,dy\quad\text{for
  all }\lambda\in(0,\bar r).
\]
Differentiating the above identity with respect to $\lambda$ we obtain that
\[
\int_{\partial B_\lambda}u^2\,dS=
\int_{\partial B_\lambda}v^2 |\mathrm{det\,Jac}\,\varPhi |\,dS\quad\text{for
  a.e. }\lambda\in(0,\bar r).
\]
Hence, by the continuity of $\mathcal H$,  
\[
\mathcal{H}(\lambda)=
\lambda^{-N}\int_{\partial B_\lambda}v^2 |\mathrm{det\,Jac}\,\varPhi |\,dS=
\int_{\mathbb{S}^N} v^2(\lambda\theta)|\mathrm{det\,Jac}\,\varPhi
(\lambda\theta)| dS
\quad\text{for
  all }\lambda\in(0,\bar r), 
\]
which yields \eqref{Hlambda} in view of \eqref{detJac}. 

From \eqref{PhiC1} and a change of variable it also follows that
\[
\frac{\int_{B_1}|\nabla \hat{v}^\lambda(x)|^2 dx}{\mathcal{H}(\lambda)}=\int_{B_1} |\nabla w^\lambda (y) \,\mathrm{Jac}\, \varPhi(\varPhi^{-1}(\lambda y))|^2 |\mathrm{det\,Jac\,} \varPhi^{-1}(\lambda y)| dy
\]
for all $\lambda\in (0,\bar{r})$. 
 The above identity, together with 
\eqref{Phisviluppo}-\eqref{detJac} and the boundedness in $H^1(B_1)$ of
$\{w^\lambda\}$ established in Lemma \ref{wlambdlalimitata},
implies estimate \eqref{gradvlambda}.
\end{proof}
\begin{lem}\label{Lemmablowupdellav}
  Let $v=u\circ \varPhi$ be as in \eqref{v} and let $k_0$ and $\gamma$
  be as in Lemma \ref{Lemmablow-up} (i). Then,
for every sequence $\lambda_n \rightarrow 0^+$, there
  exist a subsequence 
$\{\lambda_{n_k}\}_{k\in\mathbb{N}}$ and an eigenfunction $\psi$ of
 problem \eqref{eq:1} 
associated with the eigenvalue $\mu_{k_0}$
such that $\Vert \psi\Vert_{L^2(\mathbb{S}^N)}=1$, the convergence
\eqref{eq:29} holds and
\[
\frac{v(\lambda_{n_k}\cdot)}{\sqrt{\int_{\mathbb{S}^N}v^2(\lambda_{n_k}\theta) dS}}\rightarrow \psi
\quad\text{strongly in $L^2(\mathbb{S}^N)$}.
\]
\end{lem}
\begin{proof}
From Lemma \ref{Lemmablow-up}, there exist a subsequence
$\lambda_{n_k}$ and an eigenfunction 
$\psi$ of
 problem \eqref{eq:1} 
associated with the eigenvalue $\mu_{k_0}$
such that $\Vert \psi\Vert_{L^2(\mathbb{S}^N)}=1$
and \eqref{eq:29} holds. 
From  \eqref{eq:29} it follows that, up to passing to  a further subsequence,
$w^{\lambda_{n_k}}\big|_{\partial B_1}$ converges to
$\psi$ in $L^2(\mathbb{S}^N)$ and almost everywhere on
$\mathbb{S}^N$, where $w^{\lambda}$ is defined in \eqref{ulambda}. 
 From Lemma \ref{Hdellav} it follows that
$\{\hat{v}^\lambda/\sqrt{\mathcal{H}(\lambda)}\}_{\lambda}$ is bounded
in $H^1(B_1)$ and hence, in view of \eqref{Hlambda}, there exists $\tilde{\psi}\in L^2(\mathbb{S}^N)$ such that, up to a further subsequence, 
\begin{equation}\label{eq:31}
\frac{v(\lambda_{n_k}\cdot)}{\sqrt{\int_{\mathbb{S}^N}v^2(\lambda_{n_k}\theta)
    dS}}\rightarrow \tilde{\psi}\quad \text{strongly in
  $L^2(\mathbb{S}^N)$ and almost everywhere on $\sfera^N$}.
\end{equation} 
To conclude it is enough to show that $\tilde\psi=\psi$. To this aim we
observe that, for every $\varphi\in C^\infty_{\rm c}(\sfera^N)$, from
\eqref{v}, \eqref{Hlambda}, and a change of variable it follows that 
\begin{multline}\label{eq:30}
  \int_{\sfera^N}\frac{v(\lambda_{n_k}\theta)}{\sqrt{\int_{\mathbb{S}^N}v^2(\lambda_{n_k}\cdot)
  dS}}\,\varphi(\theta)\,dS\\=(1+O(\lambda_{n_k}))\int_{\sfera^N}w^{\lambda_{n_k}}(\theta)\varphi\left(\tfrac{\Phi^{-1}(\lambda_{n_k}\theta)}{\lambda_{n_k}}\right)
|\det\mathop{\rm Jac}\Phi^{-1}(\lambda_{n_k}\theta)|\,dS. 
\end{multline}
In view of \eqref{Phi-1sviluppo} and \eqref{detJac} we have that, for
all $\theta\in\sfera^N$,
\[
\lim_{k\to\infty} \varphi\left(\tfrac{\Phi^{-1}(\lambda_{n_k}\theta)}{\lambda_{n_k}}\right)
|\det\mathop{\rm Jac}\Phi^{-1}(\lambda_{n_k}\theta)|=\varphi(\theta),
\]
so that, by the Dominated Convergence Theorem, the right hand side of
\eqref{eq:30} converges to $\int_{\sfera^N}\psi(\theta)\varphi(\theta)\,dS$.
On the other hand \eqref{eq:31} implies that the left hand side of
\eqref{eq:30} converges to
$\int_{\sfera^N}\tilde\psi(\theta)\varphi(\theta)\,dS$.
Therefore, passing to the limit in \eqref{eq:30}, we obtain that 
\[
\int_{\sfera^N}\psi(\theta)\varphi(\theta)\,dS= \int_{\sfera^N}\tilde\psi(\theta)\varphi(\theta)\,dS
 \quad\text{for all }\varphi\in C^\infty_{\rm c}(\sfera^N)
\]
thus implying that $\psi=\tilde\psi$.
\end{proof}

\begin{lem}\label{lemmaliminf}
Let $k_0$ be as in Lemma \ref{Lemmablow-up}  and let
$M_{k_0}\in\mathbb{N}\setminus\{0\}$ be the multiplicity of
$\mu_{k_0}$ as an eigenvalue of \eqref{eq:1}. 
Let
$\{Y_{k_0 ,m}\}_{m=1,2,\dots ,M_{k_0}}$ be as in \eqref{eq:32}.
Then, for any sequence $\lambda_n\rightarrow 0^+$, there exists $m\in \{1,2,\dots ,M_{k_0} \}$ such that 
\[
\liminf_{n\rightarrow+\infty} \frac{\bigl|\int_{\mathbb{S}^N}v(\lambda_n\theta) Y_{k_0 ,m}(\theta)\,dS\bigr|}{\sqrt{\mathcal{H}(\lambda_n)}}>0.
\]
\end{lem}
\begin{proof}
We argue by contradiction and assume that, along a sequence $\lambda_n\rightarrow 0^+$,
\begin{equation}\label{eq:35}
\liminf_{n\rightarrow+\infty} \frac{\bigl|\int_{\mathbb{S}^N}v(\lambda_n\theta) Y_{k_0 ,m}(\theta)\,dS\bigr|}{\sqrt{\mathcal{H}(\lambda_n)}}=0
\end{equation}
for all $m\in \{1,2,\dots ,M_{k_0}\}$. From Lemma
\ref{Lemmablowupdellav} and \eqref{Hlambda}
it follows that there exist a subsequence $\{\lambda_{n_k}\}$ and an
eigenfunction $\psi$ of problem \eqref{eq:1} associated to the
eigenvalue $\mu_{k_0}$ such that $\Vert
\psi\Vert_{L^2(\mathbb{S}^N)}=1$ and
\[
\frac{v(\lambda_{n_k}\theta)}{\sqrt{\mathcal{H}(\lambda_{n_k})}}\rightarrow \psi(\theta)
\quad \text{strongly in $L^2(\mathbb{S}^N)$}.
\] 
Furthermore, from \eqref{eq:35} we have that, for every   $m\in
\{1,2,\dots ,M_{k_0}\}$, there exists a further subsequence
$\{\lambda_{n^m_k}\}$ such that 
\[
\lim_{k\rightarrow +\infty} \int_{\mathbb{S}^N}
\frac{v(\lambda_{n^m_k}\theta)}{\sqrt{\mathcal{H}(\lambda_{n^m_k})}}
Y_{k_0 ,m}(\theta)\, dS=0. 
\]
Therefore $\int_{\mathbb{S}^N}\psi\, Y_{k_0 ,m}\,dS=0$ 
for all $m\in \{1,2,\dots,M_{k_0}\}$, thus implying that $\psi\equiv
0$ and giving rise to a contradiction.
\end{proof}

For all $k\in{\mathbb N}\setminus\{0\}$, $m\in\{1,2,\dots ,M_{k}\}$, and $\lambda\in (0,\bar{r})$, we define 
\begin{equation}\label{phi}
\varphi_{k,m}(\lambda):=\int_{\mathbb{S}^N} v(\lambda\theta)Y_{k,m}(\theta)\,dS
\end{equation}
and
\begin{equation}\label{Zi}
\begin{split}
\Upsilon_{k,m}(\lambda)=-&\int_{B_\lambda}(A-\text{Id}_{N+1})\nabla v(x)\cdot\frac{\nabla_{\mathbb{S}^N} Y_{k,m} (x/|x|)}{|x|}\,dx+\int_{B_\lambda}\tilde{f}(x)v(x) Y_{k,m}(x/|x|)\,dx\\
&+\int_{\partial B_\lambda}(A-\text{Id}_{N+1})\nabla v(x)\cdot \frac{x}{|x|} Y_{k,m}(x/|x|)\,dS,
\end{split}
\end{equation}
where  the functions $\{Y_{k ,m}\}_{m=1,2,\dots ,M_{k}}$ are introduced in \eqref{eq:32}.

\begin{lem}\label{l:varphi-ups}
Let $k_0$ be as in Lemma \ref{Lemmablow-up}. For all
$m\in\{1,2,\dots,M_{k_0}\}$ and $R\in(0,\bar r]$
  \begin{equation}\label{formulaperphi-b}
\begin{split}
\varphi_{k_0,m}(\lambda)=\lambda^{\frac {k_0}2}\biggl(R^{-\frac
  {k_0}2}&\varphi_{k_0,m}(R)+\frac{2N+k_0-2}{2(N+k_0-1)}\int_\lambda^{R} s^{-N-\frac
  {k_0}2}\Upsilon_{k_0,m}(s)ds\\&+\frac{k_0\, R^{-N+1-k_0}}{2(N+k_0 -1)}\int_0^{R} s^{\frac {k_0}2-1}\Upsilon_{k_0,m}(s)\,ds\biggr)
+o(\lambda^{\frac {k_0}2})
\end{split}
\end{equation}
as $\lambda\to 0^+$.
\end{lem}
\begin{proof}
For all $k\in{\mathbb N}\setminus\{0\}$ and $m\in\{1,2,\dots ,M_{k}\}$, we consider the distribution $\zeta_{k,m}$ on $(0,\bar{r})$ defined as 
\begin{equation*}\label{zeta}
\begin{split}
  _{\mathcal{D}'(0,\bar{r})}\langle\zeta_{k,m},\omega\rangle_{\mathcal{D}(0,\bar{r})}&=\int_0^{\bar{r}} \omega(\lambda)\biggl(\int_{\mathbb{S}^N}\tilde{f}(\lambda\theta)v(\lambda\theta) Y_{k,m}(\theta)dS\biggr)d\lambda\\
  &+_{H^{-1}(B_{\bar{r}})}\bigl\langle
  \text{div}((A-\text{Id}_{N+1})\nabla v), |x|^{-N}\omega(|x|)
  Y_{k,m}(x/|x|)\bigr\rangle_{H^1_0(B_{\bar{r}})}
\end{split}
\end{equation*}
for all $\omega\in\mathcal{D}(0,\bar{r})$, where 
\begin{equation*}
_{H^{-1}(B_{\bar{r}})}\bigl\langle \text{div}((A-\text{Id}_{N+1})\nabla v), \phi\bigr\rangle_{H^1_0(B_{\bar{r}})}=-\int_{B_{\bar{r}}}(A-\text{Id}_{N+1})\nabla v\cdot\nabla \phi\,dx
\end{equation*}
for all $\phi\in H^1_0(B_{\bar{r}})$. Letting 
$\Upsilon_{k,m}$ be defined in \eqref{Zi},
we observe that $\Upsilon_{k,m}\in L^1_{\text{loc}}(0,\bar{r})$ and, by direct calculations,
\begin{equation}\label{Upsilon'=zeta}
\Upsilon'_{k,m}(\lambda)=\lambda^N\zeta_{k,m}(\lambda)\quad \text{in $\mathcal{D}'(0,\bar{r})$}.
\end{equation}
From the definition of $\zeta_{k,m}$, \eqref{eq:CPv}, and the fact
that 
    $Y_{k ,m}$ is an eigenfunction of \eqref{eq:1} associated to the
    eigenvalue $\mu_{k}$, it follows that, for 
all $k\in{\mathbb N}\setminus\{0\}$ and $m\in\{1,2,\dots ,M_{k}\}$, the function $\varphi_{k,m}$ defined in \eqref{phi} solves 
\begin{equation*}
-\varphi''_{k,m}(\lambda)-\frac{N}{\lambda}\varphi'_{k,m}(\lambda)+\frac{\mu_k}{\lambda^2}\varphi_{k,m}(\lambda)=\zeta_{k,m}(\lambda)
\end{equation*}
in the sense of distributions in $(0,\bar{r})$, which, in view of \eqref{eq:2}, can be also written as 
\begin{equation*}
-(\lambda^{N+k}(\lambda^{-\frac k2}\varphi_{k,m}(\lambda))')'=\lambda^{N+\frac k2}\zeta_{k,m}(\lambda)
\end{equation*}
in the sense of distributions in $(0,\bar{r})$.
Integrating the
right-hand side of the above equation by parts and taking into account
\eqref{Upsilon'=zeta}, we obtain that, for every
$k\in{\mathbb N}\setminus\{0\}$, $m\in\{1,2,\dots ,M_{k}\}$, and $R\in (0,\bar{r}]$, there
exists $c_{k,m}(R)\in\mathbb{R}$ such that
\begin{equation*}
(\lambda^{-\frac k2}\varphi_{k,m}(\lambda))'=-\lambda^{-N-\frac
  k2}\Upsilon_{k,m}(\lambda)-\frac
k2\lambda^{-N-k}\biggl(c_{k,m}(R)+\int_\lambda^{R} s^{\frac k2-1}\Upsilon_{k,m}(s)\,ds\biggr)
\end{equation*}
in the sense of distributions in $(0,\bar{r})$. In particular,
$\varphi_{k,m}\in W^{1,1}_{\text{loc}}(0,\bar{r})$ and, by a further integration,
\begin{equation}\label{formulaperphi}
\begin{split}
\varphi_{k,m}(\lambda)=\lambda^{\frac k2}&\biggl(R^{-\frac
  k2}\varphi_{k,m}(R)+\int_\lambda^{R}
 s^{-N-\frac k2}\Upsilon_{k,m}(s)ds\biggr)\\
&+\frac k2\lambda^{\frac k2}\int_\lambda^{R}
s^{-N-k}\biggl(c_{k,m}(R)+\int_s^{R} t^{\frac k2-1}\Upsilon_{k,m}(t)dt\biggr)ds\\
=\lambda^{\frac k2}\biggl(R^{-\frac
  k2}&\varphi_{k,m}(R)+\frac{2N+k-2}{2(N+k-1)}\int_\lambda^{R} s^{-N-\frac
  k2}\Upsilon_{k,m}(s)ds-\frac{k\, c_{k,m}(R) R^{-N+1-k}}{2(N+k -1)}\biggr)\\
&+\frac{k\,\lambda^{-N+1-\frac
    k2}}{2(N-1+k)}\biggl(c_{k,m}(R)+\int_\lambda^{R} t^{\frac k2-1}\Upsilon_{k,m}(t)dt\biggr).
\end{split}
\end{equation}
Let now $k_0$ be as in Lemma \ref{Lemmablow-up}.
We claim that 
\begin{equation}\label{eq:33}
\text{the function 
$s\mapsto s^{-N-\frac{k_0}2}\Upsilon_{k_0,m}(s)$ 
belongs to $L^1(0,\bar{r})$ for any $m\in\{1,2,\dots,M_{k_0}\}$.}
\end{equation}
To this purpose, let us estimate each term in \eqref{Zi}. By
\eqref{tildeA}, \eqref{gradvlambda}, Lemma \ref{wlambdlalimitata}, the
H\"{o}lder inequality and a change of variable we obtain that, for all $s\in(0,\bar r)$,
\begin{equation}\label{F1}
\begin{split}
  \biggl|\int_{B_s}(A(x)-\text{Id}_{N+1}&) \nabla v(x)\cdot
  \frac{\nabla_{\mathbb{S}^N}Y_{k_0,m}\big(\frac{x}{|x|}\big)}{|x|}dx\biggr|\leq \mathrm{const}\,\int_{B_s}|x||\nabla v(x)|\frac{|\nabla_{\mathbb{S}^N}Y_{k_0,m}\big (\frac{x}{|x|}\big)|}{|x|}dx\\
  &\leq \mathrm{const}\,\sqrt{\int_{B_s}|\nabla v(x)|^2 dx}\,\sqrt{\int_{B_s}|\nabla_{\mathbb{S}^N}Y_{k_0,m}\big (\tfrac{x}{|x|}\big)|^2dx}\\
  &\leq \mathrm{const}\,s^{\frac{N-1}{2}}\, s^{\frac{N+1}{2}}\sqrt{\mathcal{H}(s)}\sqrt{\int_{B_1}\frac{|\nabla \hat{v}^s(x)|^2}{\mathcal{H}(s)}dx}\leq\mathrm{const}\,s^N \sqrt{\mathcal{H}(s)}.
\end{split}
\end{equation}
By the H\"{o}lder inequality, \eqref{v}, \eqref{PhiC1}, and the definition of
$\tilde{f}$ in \eqref{eq:def-f} we have that,  
\begin{align*}
\biggl|\int_{B_s}
  \tilde{f}(x)v(x)Y_{k_0,m}\big(\tfrac{x}{|x|}\big)\,dx\biggr|&\leq
                                                                \sqrt{\int_{B_s}|\tilde{f}(x)|v^2(x)\,dx}
\sqrt{\int_{B_s}|\tilde{f}(x)|Y_{k_0,m}^2\big(\tfrac{x}{|x|}\big)\,dx}\\
&=\sqrt{\int_{B_s}|f(y)|u^2(y)\,dy}\sqrt{\int_{B_s}|f(y)|Y_{k_0,m}^2\big(\tfrac{\Phi^{-1}(y)}{|\Phi^{-1}(y)|}\big)\,dy}.
\end{align*}
From \eqref{eta}, \eqref{nablaeta}, \eqref{<1/2}, \eqref{Nlimitata},
and \eqref{primoterm} it follows that 
\begin{equation*}
\int_{B_s} |f|u^2\,dx\leq{\rm const}\, \beta(s,f)s^{N-1}\mathcal
H(s)
\end{equation*}
where $\beta(s,f)=\eta(s,f)$ under assumptions \eqref{eta0}-\eqref{eta} and $\beta(s,f)=\xi_f(s)$ under assumptions \eqref{xi0}-\eqref{xi}.
Moreover, by \eqref{eta}, \eqref{fu^2} and direct calculations we also have that 
\begin{equation*}
\int_{B_s}|f(y)|Y_{k_0,m}^2\big(\tfrac{\Phi^{-1}(y)}{|\Phi^{-1}(y)|}\big)\,dy
\leq {\rm const}\, \beta(s,f)s^{N-1}.
\end{equation*}
Therefore we conclude that, for all $s\in(0,\bar r)$, 
\begin{equation}\label{F2}
\biggl|\int_{B_s}
\tilde{f}(x)v(x)Y_{k_0,m}\big(\tfrac{x}{|x|}\big)\,dx\biggr|\leq
\mathrm{const}\,\beta(s,f)s^{N-1}\sqrt{\mathcal{H}(s)}.
\end{equation}
As regards the last term in \eqref{Zi}, we observe that, for a.e. $s\in(0,\bar r)$,  
\begin{equation}\label{F3}
\biggl|\int_{\partial B_s} (A-\text{Id}_{N+1})\nabla
v(x)\cdot\frac{x}{|x|}
Y_{k_0,m}\big(\tfrac{x}{|x|}\big)\,dS \biggr|\leq \mathrm{const}\,
s\int_{\partial B_s} |\nabla v|
\big|Y_{k_0,m}\big(\tfrac{x}{|x|}\big)\big|dS,
\end{equation}
as a consequence of \eqref{tildeA}. Integrating by parts and using
\eqref{gradvlambda}, Lemma \ref{wlambdlalimitata}, the H\"{o}lder
inequality and a change of variable we have that, for every $R\in(0,\bar r]$,
\begin{equation}\label{F3'}
\begin{split}
\int_0^{R} s^{-N-\frac{k_0}2+1}\biggl(&\int_{\partial B_s}
|\nabla v||
Y_{k_0,m}\big(\tfrac{x}{|x|}\big)
| dS\biggr)ds
=R^{-N-\frac{k_0}2+1}\int_{B_{R}} |\nabla
v|\big|Y_{k_0,m}\big(\tfrac{x}{|x|}\big)\big|dx\\
& +\big(N+\tfrac{k_0}2-1\big)\int_0^{R}
s^{-N-\frac{k_0}2}\biggl(\int_{B_s} |\nabla v|
\big|Y_{k_0,m}\big(\tfrac{x}{|x|}\big)\big|
\,dx\biggr)ds\\
&\qquad\leq\mathrm{const}\,\bigg(
R^{-\frac{k_0}2+1}\sqrt{\mathcal{H}(R)}
+\int_0^{R} s^{-\frac{k_0}2}\sqrt{\mathcal{H}(s)}ds\bigg).
\end{split}
\end{equation}
 From \eqref{Zi},  \eqref{F1}, \eqref{F2}, \eqref{F3}, and \eqref{F3'}
we deduce that, for all  $m\in\{1,2,\dots,M_{k_0}\}$ and $R\in(0,\bar r]$, 
\begin{equation}\label{eq:34}
  \int_0^R s^{-N-\frac{k_0}2}|\Upsilon_{k_0,m}(s)|\,ds\leq
  \mathrm{const}\,R^{-\frac{k_0}2+1}\sqrt{\mathcal{H}(R)}+\int_0^{R} s^{-\frac{k_0}2}\sqrt{\mathcal{H}(s)}\left(1+s^{-1}\beta(s,f)\right)ds.
\end{equation}
Thus claim 
\eqref{eq:33}
follows from \eqref{eq:34}, \eqref{H<kr} and assumptions \eqref{xiL1} and \eqref{etaL1}. 

 From \eqref{eq:33} we deduce that, for every fixed $R\in(0,\bar r]$,
\begin{equation}\label{opiccololambda-N+1-sigma}
\begin{split}
\lambda^{\frac {k_0}2}\biggl(R^{-\frac
  {k_0}2}\varphi_{{k_0},m}(R)&+\frac{2N+{k_0}-2}{2(N+{k_0}-1)}\int_\lambda^{R} s^{-N-\frac
  {k_0}2}\Upsilon_{{k_0},m}(s)ds-\frac{{k_0}\, c_{{k_0},m}(R) R^{-N+1-{k_0}}}{2(N+{k_0} -1)}\biggr)\\
&=O(\lambda^{\frac{k_0}2})=o(\lambda^{-N+1-\frac{k_0}2})\quad \text{as $\lambda\rightarrow 0^+$}.
\end{split}
\end{equation}
On the other hand, \eqref{eq:33} also implies that
$t \mapsto t^{\frac{k_0}2 -1}\Upsilon_{k_0,m}(t)\in
L^1(0,\bar{r})$. We claim that,   for every $R\in(0,\bar r]$,
\begin{equation}\label{ci+int=0}
c_{k_0,m}(R)+\int_0^{R} t^{\frac {k_0}2-1}\Upsilon_{k_0,m}(t)dt=0.
\end{equation}
Suppose by contradiction that \eqref{ci+int=0} is not true for some $R\in(0,\bar r]$. Then, from \eqref{formulaperphi} and \eqref{opiccololambda-N+1-sigma} we infer that
\begin{equation}\label{2}
\varphi_{k_0,m}(\lambda)\sim \frac{k_0\,\lambda^{-N+1-\frac
    {k_0}2}}{2(N-1+k_0)}\biggl(c_{k_0,m}(R)+\int_0^{R} t^{\frac {k_0}2-1}\Upsilon_{k_0,m}(t)dt\biggr) \quad \text{as $\lambda\rightarrow 0^+$}. 
\end{equation}
Lemma \ref{lemmaHardy} and the fact that $v\in H^1(B_{\bar{r}})$ imply that
\begin{equation*}
\int_0^{\bar{r}} \lambda^{N-2}|\varphi_{k_0,m}(\lambda)|^2 \,d\lambda\leq \int_0^{\bar{r}} \lambda^{N-2}\biggl(\int_{\mathbb{S}^N}|v(\lambda\theta)|^2 dS\biggr)d\lambda=\int_{B_{\tilde{r}}}\frac{|v(x)|^2}{|x|^2}dx<+\infty,
\end{equation*}
thus contradicting \eqref{2}. Claim \eqref{ci+int=0} is thereby
proved.

 From \eqref{eq:33} and \eqref{ci+int=0} it follows that,   for every $R\in(0,\bar r]$, 
\begin{equation}\label{lambda^sigmai-}
\begin{split}\bigg|&\lambda^{-N+1-\frac
    {k_0}2}\biggl(c_{k_0,m}(R)+\int_\lambda^{R} t^{\frac
    {k_0}2-1}\Upsilon_{k_0,m}(t)dt\biggr)\bigg|=
\lambda^{-N+1-\frac
    {k_0}2}\bigg|\int_0^\lambda t^{\frac
    {k_0}2-1}\Upsilon_{k_0,m}(t)dt
\bigg|\\
&\leq
\lambda^{-N+1-\frac
    {k_0}2}\int_0^\lambda t^{N+k_0-1} \Big|t^{
    -N-\frac{k_0}2}\Upsilon_{k_0,m}(t)\Big|dt
\leq
\lambda^{\frac
    {k_0}2}\int_0^\lambda \Big|t^{
    -N-\frac{k_0}2}\Upsilon_{k_0,m}(t)\Big|dt=o(\lambda^{\frac
    {k_0}2})
\end{split}
\end{equation}
as $\lambda\rightarrow 0^+$. 

The conclusion follows by combining \eqref{formulaperphi},
\eqref{lambda^sigmai-}, and \eqref{ci+int=0}.
\end{proof}

 \begin{lem} \label{LemmalimH>0}
Let $\gamma$ be as in Lemma \ref{Lemlimesiste}. Then
$\lim_{r\rightarrow 0^+} r^{-2\gamma}\mathcal{H}(r)>0$.
\end{lem}
\begin{proof}
For any $\lambda\in (0,\bar{r})$, we expand $\theta\mapsto
v(\lambda\theta)\in L^2(\mathbb{S}^N)$ in Fourier series with respect
to the orthonormal basis
    $\{Y_{k ,m}\}_{m=1,2,\dots ,M_{k}}$ introduced in \eqref{eq:32}, i.e. 
\begin{equation}\label{vespansion}
v(\lambda\theta)=\sum_{k=1}^\infty\sum_{m=1}^{M_k} \varphi_{k,m}(\lambda) Y_{k ,m}(\theta)\quad \text{in}\ L^2(\mathbb{S}^N),
\end{equation}
where, for all  $k\in{\mathbb N}\setminus\{0\}$, $m\in\{1,2,\dots
,M_{k}\}$, and $\lambda\in (0,\bar{r})$,
$\varphi_{k,m}(\lambda)$ is defined in \eqref{phi}.

Let $k_0\in\mathbb{N}$, $k_0\geq 1$, be as in 
 Lemma \ref{Lemmablow-up}, so that
\begin{equation}\label{gammasigmai+}
 \gamma=\lim_{r\rightarrow 0^+}\mathcal{N}(r)=\frac{k_0}2.
\end{equation}
From \eqref{Hlambda} and the Parseval identity we deduce that 
\begin{equation}\label{parseval}
\mathcal{H}(\lambda)=(1+O(\lambda))\int_{\mathbb{S}^N}v^2(\lambda\theta)\,
dS=(1+O(\lambda))\sum_{k=1}^\infty\sum_{m=1}^{M_k}\varphi_{k,m}^2(\lambda), 
\end{equation}
for all $0<\lambda\leq  \bar{r}$.  Let us assume by contradiction that $\lim_{\lambda\rightarrow 0^+}\lambda^{-2\gamma}\mathcal{H}(\lambda)=0$. Then, \eqref{gammasigmai+} and \eqref{parseval} imply that 
\begin{equation}\label{limphii}
\lim_{\lambda\rightarrow 0^+} \lambda^{-k_0/2}\varphi_{k_0,m}(\lambda)=0\quad \text{for any $m\in\{1,2,\dots,M_{k_0}\}$}.
\end{equation}
From \eqref{formulaperphi-b} and  \eqref{limphii} we obtain that
\begin{equation}\label{c1=0}
  \begin{split}
    R^{-\frac
      {k_0}2}\varphi_{k_0,m}(R)&+\frac{2N+k_0-2}{2(N+k_0-1)}\int_0^{R}
    s^{-N-\frac {k_0}2}\Upsilon_{k_0,m}(s)ds\\&\quad+\frac{k_0\,
      R^{-N+1-k_0}}{2(N+k_0 -1)}\int_0^{R} s^{\frac
      {k_0}2-1}\Upsilon_{k_0,m}(s)\,ds=0
  \end{split}
\end{equation}
for all $R\in (0, \bar{r}]$ and $m\in\{1,2,\dots,M_{k_0}\}$. 

Since we are assuming by contradiction that $\lim_{\lambda\rightarrow 0^+}\lambda^{-2\gamma}\mathcal{H}(\lambda)=0$,
there exists a sequence $\{R_n\}_{n\in\mathbb{N}}\subset (0, \bar{r})$
such that $R_{n+1}<R_n$, $\lim_{n\to\infty}R_n=0$ and 
\begin{equation*}
R_n^{-k_0/2}\sqrt{\mathcal{H}(R_n)}=\max_{s\in[0,R_n]}\left(s^{-k_0/2}\sqrt{\mathcal{H}(s)}\right).
\end{equation*}
By Lemma \ref{lemmaliminf} with $\lambda_n=R_n$, there exists
$m_0\in \{1,2,\dots,M_{k_0}\}$ such that, up to a subsequence,
\begin{equation}\label{limdivda0}
\lim _{n\rightarrow\infty}\frac{\varphi_{k_0,m_0}(R_n)}{\sqrt{\mathcal{H}(R_n)}}\neq 0.
\end{equation}
  By \eqref{c1=0}, \eqref{eq:34}, \eqref{limdivda0}, \eqref{H<kr},
\eqref{xiL1} and \eqref{etaL1},
 we have 
\begin{equation}\label{o1}
\begin{split}
  &\bigg| R_n^{-\frac {k_0}2}\varphi_{k_0,m_0}(R_n)
+\frac{k_0\,
      R_n^{-N+1-k_0}}{2(N+k_0 -1)}\int_0^{R_n} s^{\frac
      {k_0}2-1}\Upsilon_{k_0,m_0}(s)\,ds\bigg|\\
&=\bigg|
  \frac{2N+{k_0}-2}{2(N+{k_0}-1)}\int_0^{R_n} s^{-N-\frac
    {k_0}2}\Upsilon_{{k_0},m_0}(s)ds\bigg|\\
&\leq
  \frac{2N+{k_0}-2}{2(N+{k_0}-1)}\int_0^{R_n} s^{-N-\frac
    {k_0}2}|\Upsilon_{{k_0},m_0}(s)|ds\\
  &\leq {\rm const\,} \biggl(
R_n^{-\frac{k_0}2+1}\sqrt{\mathcal{H}(R_n)}+
\int_0^{R_n}
  s^{-\frac{k_0}2}\sqrt{\mathcal{H}(s)}\Big(1+s^{-1}\beta(s,f)\Big)\,ds\biggr)\\
  &\leq {\rm
    const\,}\biggl(R_n^{-\frac{k_0}2}\sqrt{\mathcal{H}(R_n)}R_n+
  R_n^{-\frac{k_0}2}\sqrt{\mathcal{H}(R_n)}
  \int_0^{R_n}\frac{\beta(s,f)}{s}ds\biggr)\\
  &\leq {\rm
    const\,}\biggl( \biggl|\frac{\sqrt{\mathcal{H}(R_n)}}{\varphi_{k_0,m_0}(R_n)}\biggr|\biggl|\frac{\varphi_{k_0,m_0} (R_n)}{R_n^{k_0/2}}\biggr|R_n+\biggl|\frac{\sqrt{\mathcal{H}(R_n)}}{\varphi_{k_0,m_0}(R_n)}\biggr|\biggl|\frac{\varphi_{k_0,m_0}(R_n)}{R_n^{k_0/2}}\biggr|\int_0^{R_n}\frac{\beta(s,f)}{s}ds\biggr)\\
  &=o\biggl(\frac{\varphi_{k_0,m_0}(R_n)}{R_n^{k_0/2}}\biggr)
\end{split}
\end{equation}
as $n\rightarrow +\infty$.    
On the other hand, by \eqref{o1} we also have that
\begin{equation}\label{o2}
 \ \begin{split}
   \frac{k_0\, R_n^{-N+1-k_0}}{2(N+k_0 -1)}&\left|
\int_0^{R_n} t^{\frac {k_0}2-1}\Upsilon_{k_0,m_0}(t)dt\right|\\
&=    \frac{k_0\, R_n^{-N+1-k_0}}{2(N+k_0 -1)}\left|
\int_0^{R_n} t^{N+k_0-1} t^{-N-\frac
  {k_0}2}\Upsilon_{k_0,m_0}(t)dt\right|  \\
&\leq  \frac{k_0}{2(N+k_0 -1)}
\int_0^{R_n} t^{-N-\frac {k_0}2}|\Upsilon_{k_0,m_0}(t)|dt =o\biggl(\frac{\varphi_{k_0,m_0}(R_n)}{R_n^{k_0/2}}\biggr) 
\end{split}
\end{equation}
as $n\rightarrow +\infty$. 
Combining \eqref{o1} with \eqref{o2} we obtain that 
\[
R_n^{-\frac
  {k_0}2}\varphi_{k_0,m_0}(R_n)=o\Bigl(R_n^{-\frac
  {k_0}2}\varphi_{k_0,m_0}(R_n)\Bigr)
\quad\text{as }n\rightarrow +\infty,
\]
which is a contradiction. 
\end{proof}
Combining Lemma \ref{Lemmablow-up}, Lemma \ref{Lemmablowupdellav} and Lemma \ref{LemmalimH>0}, we can now prove the following theorem which is a more precise and complete version of Theorem \ref{teoremagenerale}. 
\begin{thm}\label{teoremabetai}
Let $N\geq 2$ and $u\in H^1(B_{\hat{R}})\setminus \{0\}$ be a non-trivial weak solution to \eqref{eq:CPu}, with $f$ satisfying either assumptions \eqref{xi0}-\eqref{xi} or \eqref{eta0}-\eqref{eta}. Then, letting $\mathcal{N}(r)$ be as in \eqref{frequency}, there exists $k_0\in\mathbb{N}$, $k_0\geq 1$, such that
\begin{equation}\label{Nu,f=sigmai}
\lim_{r\rightarrow 0^+} \mathcal{N}(r)=\frac{k_0}{2}.
\end{equation}
Furthermore, if 
 $M_{k_0}\in\N\setminus\{0\}$ is the
multiplicity of $\mu_{k_0}$ as an eigenvalue of problem \eqref{eq:1} and 
$\{Y_{k_0 ,m}\}_{m=1,2,\dots ,M_{k_0}}$ is a
      $L^{2}(\mathbb{S}^{N})$-orthonormal basis of the eigenspace associated to $\mu_{k_0}$, then
\begin{equation}\label{convulambdabetai}
\lambda^{-k_0/2}u(\lambda x)\rightarrow |x|^{k_0/2}\sum_{m=1}^{M_{k_0}} \beta_mY_{k_0,m}\biggl(\frac{x}{|x|}\biggr) \quad \text{in $H^1(B_1)$} \quad \text{as $\lambda\rightarrow 0^+$},
\end{equation}
where $(\beta_{1},\beta_{2},\dots,\beta_{M_{k_0}})\neq (0,0,\dots,0)$ 
and
\begin{equation}\label{betaiespliciti}
\begin{split}
\beta_m = \int_{\mathbb{S}^N} &R^{-k_0/2} u(\varPhi (R\theta))Y_{k_0,m}(\theta) dS\\
& +\frac{1}{1-N-k_0}\int_0^R\biggl(\frac{1-N-\frac{k_0}2}{s^{N+\frac{k_0}2}}-\frac{k_0 \,s^{\frac{k_0}2-1}}{2R^{N-1+k_0}}\biggr) \Upsilon_{k_0,m}(s)\, ds
\end{split}
\end{equation}
for all $R\in (0,\bar{r})$ for some $\bar{r}>0$, where
$\Upsilon_{k_0,m}$ is defined in \eqref{Zi} and $\Phi$ is the
diffeomorphism introduced in Lemma \ref{Straightening}.
\end{thm}
\begin{proof}
 Identity \eqref{Nu,f=sigmai} follows immediately from Lemma
 \ref{Lemmablow-up}. 

 In order to prove \eqref{convulambdabetai}, let
 $\{\lambda_n\}_{n\in\mathbb{N}}\subset (0,\infty)$ be such
 that $\lambda_n\rightarrow 0^+$ as $n\rightarrow +\infty$. By Lemmas
 \ref{Lemmablow-up}, \ref{LemmaesistelimH}, \ref{Lemmablowupdellav},
 \ref{LemmalimH>0} and \eqref{Hlambda}, there exist a subsequence
 $\{\lambda_{n_j}\}_j$ and constants
 $\beta_{1},\beta_{2},\dots,\beta_{M_{k_0}}\in\R$ such
 that
 $(\beta_{1},\beta_{2},\dots,\beta_{M_{k_0}})\neq
 (0,0,\dots,0)$,
\begin{equation}\label{ulambdabetai}
\lambda_{n_j}^{-\frac{k_0}2}u(\lambda_{n_j}x)\rightarrow |x|^{\frac{k_0}2}\sum_{m=1}^{M_{k_0}}\beta_mY_{k_0,m}\biggl(\frac{x}{|x|}\biggr) \quad \text{in $H^1(B_1)$} \quad\text{as $j\rightarrow +\infty$}
\end{equation}
and
\begin{equation}\label{5}
\lambda_{n_j}^{-\frac{k_0}2}v(\lambda_{n_j}\cdot)\rightarrow 
\sum_{m=1}^{M_{k_0}}\beta_mY_{k_0,m} \quad \text{in $L^2(\mathbb{S}^N)$} \quad\text{as $j\rightarrow +\infty$}.
\end{equation}
We will now prove that the $\beta_m$'s depend neither on the sequence
$\{\lambda_n\}_{n\in\mathbb{N}}$ nor on its subsequence
$\{\lambda_{n_j}\}_{j\in\mathbb{N}}$. Let us fix $R\in (0,\bar{r})$,
with $\bar{r}$ as in Lemma \ref{Straightening}, and define
$\varphi_{k_0,m}$ as in \eqref{phi}. From \eqref{5} it follows that,
for any $m=1,2,\dots, M_{k_0}$,
\begin{equation}\label{lambdanj}
\lim_{j\rightarrow +\infty}\lambda_{n_j}^{-\frac{k_0}2}\varphi_{k_0,m}(\lambda_{n_j})=\lim_{j\rightarrow +\infty}\int_{\mathbb{S}^N}\frac{v(\lambda_{n_j}\theta)}{\lambda_{n_j}^{k_0/2}}
Y_{k_0,m}(\theta) dS= \sum_{i=1}^{M_{k_0}}\beta_i\int_{\mathbb{S}^N}Y_{k_0,i}\, Y_{k_0,m}dS=\beta_m.
\end{equation}
On the other hand, \eqref{formulaperphi-b} implies that, for any $m=1,2,\dots, M_{k_0}$, 
\begin{multline*}
\lim_{\lambda \rightarrow 0^+}\lambda^{-\frac{k_0}2}\varphi_{k_0,m}(\lambda)=R^{-\frac
  {k_0}2}\varphi_{k_0,m}(R)+\frac{2N+k_0-2}{2(N+k_0-1)}\int_0^{R} s^{-N-\frac
  {k_0}2}\Upsilon_{k_0,m}(s)ds\\+\frac{k_0\, R^{-N+1-k_0}}{2(N+k_0 -1)}\int_0^{R} s^{\frac {k_0}2-1}\Upsilon_{k_0,m}(s)\,ds,
\end{multline*}
with $\Upsilon_{k_0,m}$ as in \eqref{Zi}, and therefore from \eqref{lambdanj} we deduce that 
\begin{multline*}
\beta_m= 
R^{-\frac
  {k_0}2}\varphi_{k_0,m}(R)+\frac{2N+k_0-2}{2(N+k_0-1)}\int_0^{R} s^{-N-\frac
  {k_0}2}\Upsilon_{k_0,m}(s)ds\\+\frac{k_0\, R^{-N+1-k_0}}{2(N+k_0 -1)}\int_0^{R} s^{\frac {k_0}2-1}\Upsilon_{k_0,m}(s)\,ds
\end{multline*}
for any $m=1,2,\dots,M_{k_0}$. In particular the $\beta_m$'s depend neither on the sequence $\{\lambda_n\}_{n\in\mathbb{N}}$ nor on its subsequence $\{\lambda_{n_k}\}_{k\in\mathbb{N}}$, thus implying that the convergence in \eqref{ulambdabetai} actually holds as $\lambda\rightarrow 0^+$, and proving the theorem.
\end{proof}

\appendix 
\section{Eigenvalues of problem \eqref{eq:1}}\label{sec:app_A}

In this appendix, we derive the explicit formula \eqref{eq:2} for the
eigenvalues of problem \eqref{eq:1}. 

Let us start by observing that, if $\mu$ is an eigenvalue of
\eqref{eq:1} with an associated eigenfunction $\psi$, then,
letting $\sigma=-\frac{N-1}2+\sqrt{\big(\frac{N-1}2\big)^2+\mu}$,
 the function  
$W(\rho\theta)=\rho^\sigma\psi(\theta)$ belongs to
$H^1_{\tilde{\Gamma}}(B_1)$ and is harmonic in $B_1\setminus
\tilde{\Gamma}$. From \cite{Costabel} it follows that there exists
$k\in\mathbb{N}\setminus\{0\}$ such that $\sigma=\frac k2$, so that
$\mu=\frac k4(k+2N-2)$. Moreover, from \cite{Costabel} we also deduce
that $W\in L^\infty(B_1)$, thus implying that $\psi\in
L^\infty(\sfera^N)$. 

Viceversa, let us prove that all
numbers of the form $\mu=\frac k4(k+2N-2)$ with
$k\in\mathbb{N}\setminus\{0\}$  
are eigenvalues of \eqref{eq:1}. Let us fix $k\in\mathbb{N}\setminus\{0\}$  and
consider the function $W$ defined, in cylindrical coordinates, as  
\[
W(x',r\cos t,r\sin t)=r^{k/2}\sin\bigg(\frac k2 \,t\bigg),\quad x'\in
\R^{N-1},\ r\geq 0,\  t\in [0,2\pi].
\]
We have that $W$ belongs to $H^1_{\tilde{\Gamma}}(B_1)$ and is harmonic in $B_1\setminus
\tilde{\Gamma}$; furthermore $W$ is homogeneous of degree $k/2$, so
that, letting $\psi:=W\big|_{\sfera^N}$, we have that $\psi\in
H^1_0(\sfera^N\setminus\Sigma)$, $\psi\not\equiv0$, and 
\begin{equation}\label{eq:3}
W(\rho\theta)=\rho^{k/2}\psi(\theta),\quad \rho\geq0,\ \theta\in
\sfera^N.
\end{equation}
Plugging \eqref{eq:3} into the equation $\Delta W=0$ in $B_1\setminus
\tilde{\Gamma}$, we obtain that 
\[
\rho^{\frac k2-2}\Big(
\tfrac k2\big(\tfrac k2-1+N\big)\psi(\theta)
+\Delta_{\sfera^N}\psi\Big)=0,\quad \rho>0,\ \theta\in \sfera^N\setminus\Sigma,
\]
so that $\frac k4(k+2N-2)$ is an eigenvalue of \eqref{eq:1}. 

We then conclude that the set of all eigenvalues of problem
\eqref{eq:1} is 
$\left\{
\frac{k(k+2N-2)}4:\, k\in \mathbb{N}\setminus\{0\}\right\}$
and all eigenfunctions belong to $L^\infty(\sfera^N)$. 

We observe in particular that the first eigenvalue $\mu_1=\frac
{2N-1}4$ is simple and an  associated eigenfunction is given by the
function 
\[
\Phi(\theta',\theta_N,\theta_{N+1})=\sqrt{
\sqrt{\theta_N^2+\theta_{N+1}^2}-\theta_N},\quad
(\theta',\theta_N,\theta_{N+1})\in \mathbb{S}^N.
\]


\begin{thebibliography}{10}
\bibliographystyle{acm} 


\bibitem{Adolf}
{\sc Adolfsson, V.}
\newblock {$L^2$}-integrability of second-order derivatives for {P}oisson's
  equation in nonsmooth domains.
\newblock {\em Math. Scand. 70}, 1 (1992), 146--160.

\bibitem{Adolf-Escau}
{\sc Adolfsson, V., and Escauriaza, L.}
\newblock {$C^{1,\alpha}$} domains and unique continuation at the boundary.
\newblock {\em Comm. Pure Appl. Math. 50}, 10 (1997), 935--969.

\bibitem{Adolf-Kenig}
{\sc Adolfsson, V., Escauriaza, L., and Kenig, C.}
\newblock Convex domains and unique continuation at the boundary.
\newblock {\em Rev. Mat. Iberoamericana 11}, 3 (1995), 513--525.

\bibitem{almgren}
{\sc Almgren, Jr., F.~J.}
\newblock {$Q$} valued functions minimizing {D}irichlet's integral and the
  regularity of area minimizing rectifiable currents up to codimension two.
\newblock {\em Bull. Amer. Math. Soc. (N.S.) 8}, 2 (1983), 327--328.

\bibitem{Bern}
{\sc Bernard, J.-M.~E.}
\newblock Density results in {S}obolev spaces whose elements vanish on a part
  of the boundary.
\newblock {\em Chin. Ann. Math. Ser. B 32}, 6 (2011), 823--846.

\bibitem{Carleman}
{\sc Carleman, T.}
\newblock Sur un probl\`eme d'unicit\'{e} pur les syst\`emes d'\'{e}quations
  aux d\'{e}riv\'{e}es partielles \`a deux variables ind\'{e}pendantes.
\newblock {\em Ark. Mat., Astr. Fys. 26}, 17 (1939), 9.

\bibitem{CD}
{\sc Chkadua, O., and Duduchava, R.}
\newblock Asymptotics of functions represented by potentials.
\newblock {\em Russ. J. Math. Phys. 7}, 1 (2000), 15--47.

\bibitem{Costabel}
{\sc Costabel, M., Dauge, M., and Duduchava, R.}
\newblock Asymptotics without logarithmic terms for crack problems.
\newblock {\em Comm. Partial Differential Equations 28}, 5-6 (2003), 869--926.

\bibitem{DalMaso}
{\sc Dal~Maso, G., Orlando, G., and Toader, R.}
\newblock Laplace equation in a domain with a rectilinear crack: higher order
  derivatives of the energy with respect to the crack length.
\newblock {\em NoDEA Nonlinear Differential Equations Appl. 22}, 3 (2015),
  449--476.

\bibitem{Daners2003}
{\sc Daners, D.}
\newblock Dirichlet problems on varying domains.
\newblock {\em J. Differential Equations 188}, 2 (2003), 591--624.

\bibitem{DFV}
{\sc Dipierro, S., Felli, V., and Valdinoci, E.}
\newblock Unique continuation principles in cones under nonzero {N}eumann
  boundary conditions.
\newblock {\em Annales de l'Institut Henri Poincaré C, Analyse non
  linéaire 37},
4 (2020), 785--815.

\bibitem{DW}
{\sc Duduchava, R., and Wendland, W.~L.}
\newblock The {W}iener-{H}opf method for systems of pseudodifferential
  equations with an application to crack problems.
\newblock {\em Integral Equations Operator Theory 23}, 3 (1995), 294--335.

\bibitem{FGL}
{\sc Fabes, E.~B., Garofalo, N., and Lin, F.-H.}
\newblock A partial answer to a conjecture of {B}. {S}imon concerning unique
  continuation.
\newblock {\em J. Funct. Anal. 88}, 1 (1990), 194--210.

\bibitem{Fall}
{\sc Fall, M.~M., Felli, V., Ferrero, A., and Niang, A.}
\newblock Asymptotic expansions and unique continuation at
  {D}irichlet--{N}eumann boundary junctions for planar elliptic equations.
\newblock {\em Mathematics in Engineering 1\/} (2019), 84--117.

\bibitem{Almgren-type}
{\sc Felli, V., and Ferrero, A.}
\newblock Almgren-type monotonicity methods for the classification of behaviour
  at corners of solutions to semilinear elliptic equations.
\newblock {\em Proc. Roy. Soc. Edinburgh Sect. A 143}, 5 (2013), 957--1019.

\bibitem{Semilinear}
{\sc Felli, V., and Ferrero, A.}
\newblock On semilinear elliptic equations with borderline {H}ardy potentials.
\newblock {\em J. Anal. Math. 123\/} (2014), 303--340.

\bibitem{Asymptotic}
{\sc Felli, V., Ferrero, A., and Terracini, S.}
\newblock Asymptotic behavior of solutions to {S}chr\"{o}dinger equations near
  an isolated singularity of the electromagnetic potential.
\newblock {\em J. Eur. Math. Soc. (JEMS) 13}, 1 (2011), 119--174.

\bibitem{MilanJ}
{\sc Felli, V., Ferrero, A., and Terracini, S.}
\newblock A note on local asymptotics of solutions to singular elliptic
  equations via monotonicity methods.
\newblock {\em Milan J. Math. 80}, 1 (2012), 203--226.

\bibitem{Lemmini}
{\sc Felli, V., Ferrero, A., and Terracini, S.}
\newblock On the behavior at collisions of solutions to {S}chr\"{o}dinger
  equations with many-particle and cylindrical potentials.
\newblock {\em Discrete Contin. Dyn. Syst. 32}, 11 (2012), 3895--3956.

\bibitem{Garofalo}
{\sc Garofalo, N., and Lin, F.-H.}
\newblock Monotonicity properties of variational integrals, {$A_p$} weights and
  unique continuation.
\newblock {\em Indiana Univ. Math. J. 35}, 2 (1986), 245--268.

\bibitem{Kassman}
{\sc Kassmann, M., and Madych, W.~R.}
\newblock Difference quotients and elliptic mixed boundary value problems of
  second order.
\newblock {\em Indiana Univ. Math. J. 56}, 3 (2007), 1047--1082.

\bibitem{KLH}
{\sc Khludnev, A., Leontiev, A., and Herskovits, J.}
\newblock Nonsmooth domain optimization for elliptic equations with unilateral
  conditions.
\newblock {\em J. Math. Pures Appl. (9) 82}, 2 (2003), 197--212.

\bibitem{Kukavica}
{\sc Kukavica, I.}
\newblock Quantitative uniqueness for second-order elliptic operators.
\newblock {\em Duke Math. J. 91}, 2 (1998), 225--240.

\bibitem{Kukavica-Nystrom}
{\sc Kukavica, I., and Nystr\"{o}m, K.}
\newblock Unique continuation on the boundary for {D}ini domains.
\newblock {\em Proc. Amer. Math. Soc. 126}, 2 (1998), 441--446.

\bibitem{Lazzaroni}
{\sc Lazzaroni, G., and Toader, R.}
\newblock Energy release rate and stress intensity factor in antiplane
  elasticity.
\newblock {\em J. Math. Pures Appl. (9) 95}, 6 (2011), 565--584.

\bibitem{Mosco1969}
{\sc Mosco, U.}
\newblock Convergence of convex sets and of solutions of variational
  inequalities.
\newblock {\em Advances in Math. 3\/} (1969), 510--585.

\bibitem{Savare}
{\sc Savar\'{e}, G.}
\newblock Regularity and perturbation results for mixed second order elliptic
  problems.
\newblock {\em Comm. Partial Differential Equations 22}, 5-6 (1997), 869--899.

\bibitem{TZ05}
{\sc Tao, X., and Zhang, S.}
\newblock Boundary unique continuation theorems under zero {N}eumann boundary
  conditions.
\newblock {\em Bull. Austral. Math. Soc. 72}, 1 (2005), 67--85.

\bibitem{Tao}
{\sc Tao, X., and Zhang, S.}
\newblock Weighted doubling properties and unique continuation theorems for the
  degenerate {S}chr\"{o}dinger equations with singular potentials.
\newblock {\em J. Math. Anal. Appl. 339}, 1 (2008), 70--84.

\bibitem{Wang}
{\sc Wang, Z.-Q., and Zhu, M.}
\newblock Hardy inequalities with boundary terms.
\newblock {\em Electron. J. Differential Equations\/} (2003), No. 43, 8.

\bibitem{Wolff}
{\sc Wolff, T.~H.}
\newblock A property of measures in {${\mathbb R}^N$} and an application to
  unique continuation.
\newblock {\em Geom. Funct. Anal. 2}, 2 (1992), 225--284.

\end{thebibliography}
\end{document}